\documentclass[11pt,a4paper]{article}

\usepackage{a4}

\usepackage[utf8]{inputenc}
\usepackage[english]{babel}

\usepackage{amsmath, amsthm}
\usepackage{amsfonts,amssymb}
\usepackage{mathrsfs,upgreek,euscript}
\usepackage{dsfont,xfrac,mathalfa}
\usepackage{thmtools,yhmath}
\usepackage{authblk}
\usepackage{geometry}
\usepackage[dvipsnames]{xcolor}

\usepackage{hyperref}

\usepackage{chngcntr}

\newcommand{\N}{\mathbb{N}}
\newcommand{\Z}{\mathbb{Z}}

\newcommand{\R}{\mathbb{R}}

\newcommand{\E}{\mathbb{E}}

\newcommand{\rme}{{\rm e}}

\newcommand{\rmC}{{\rm C}}
\newcommand{\rmD}{{\rm D}}
\newcommand{\rmL}{{\rm L}}
\newcommand{\1}{\mathds{1}}

\newcommand{\Fc}{\mathcal{F}}
\newcommand{\Ec}{\mathcal{E}}
\newcommand{\Gc}{\mathcal{G}}

\newcommand{\Bc}{\mathcal{B}}
\newcommand{\Mc}{\mathcal{M}}
\newcommand{\Kc}{\mathcal{K}}
\newcommand{\Uc}{\mathcal{U}}

\newcommand{\Vc}{\mathscr{V}}

\renewcommand{\mid}{~\middle | ~}
\newcommand{\eps}{\varepsilon}
\newcommand{\mycdot}{\cdot}
\newcommand{\cv}[2][]{\underset{#2}{\overset{#1}{\longrightarrow}}}
\newcommand\quotient[2]{
        \mathchoice
            {
                \text{\raise1ex\hbox{$#1$}\Big/\lower1ex\hbox{$#2$}}%
            }
            {
                #1\,/\,#2
            }
            {
                #1\,/\,#2
            }
            {
                #1\,/\,#2
            }
}
\newcommand{\cadlag}{cadlag }
\renewcommand{\P}{\mathbb{P}}
\newcommand{\Pcal}{\mathcal{P}}
\newcommand{\D}{\mathbb{D}}

\newcommand{\loc}{\text{loc}}
\newcommand{\Dloc}{\D_\loc}

\renewcommand{\d}{{\rm d}}
\newcommand{\Pbf}{\mathbf{P}}
\newcommand{\Ebf}{\mathbf{E}}

\newcommand{\law}{\mathscr{L}}

\newcommand{\Dbf}{\text{\itshape \bfseries D}}

\usepackage[many]{tcolorbox}
\newcounter{myhypo}

\newtcolorbox{hypo}[1]{
  breakable,
  enhanced,
  top=0pt,
  bottom=0pt,
  nobeforeafter,
  colback=white,
  boxrule=0pt,
  arc=0pt,
  right=40pt,
  left=20pt,
  outer arc=0pt,
  overlay={
    \node[inner sep=0pt,anchor=east] 
    at (frame.east) 
    {(#1)};
  },
}
\newenvironment{hypotheses}
  {\list{}{\setlength\leftmargin{0pt}\item\relax}}
  {\endlist}
\newcommand\Hypo[2]{%
  \begin{hypo}{#1}#2\end{hypo}}  
\makeatletter
\newcommand{\pushright}[1]{\ifmeasuring@#1\else\omit\hfill$\displaystyle#1$\fi\ignorespaces}
\newcommand{\pushleft}[1]{\ifmeasuring@#1\else\omit$\displaystyle#1$\hfill\fi\ignorespaces}
\makeatother

\theoremstyle{plain}
\newtheorem{theorem}{Theorem}[section]
\newtheorem{lemma}[theorem]{Lemma}
\newtheorem{proposition}[theorem]{Proposition}

\theoremstyle{definition}
\newtheorem{definition}[theorem]{Definition}
\newtheorem{example-base}[theorem]{Example}
\newtheorem{remark-base}[theorem]{Remark}

\theoremstyle{remark}

\counterwithin*{equation}{section}

\newenvironment{remark}{\pushQED{\qed}\begin{remark-base}}{\popQED\end{remark-base}}
\newenvironment{example}{\pushQED{\qed}\begin{example-base}}{\popQED\end{example-base}}

\begin{document}

\title{L\'evy-type processes: convergence and discrete schemes}
\author{Mihai Gradinaru and Tristan Haugomat\\
{\small Institut de Recherche Math{\'e}matique de Rennes, 
Universit{\'e} de Rennes 1},\\ 
{\small Campus de Beaulieu, 35042 Rennes Cedex, France}\\
{\tt{\small \{Mihai.Gradinaru,Tristan.Haugomat\}@univ-rennes1.fr}}}
\date{
}
\maketitle

{\small\noindent {\bf Abstract:}~ We characterise the convergence of a certain class of discrete time Markov processes toward locally Feller processes in terms of convergence  
of associated operators. The theory of locally Feller processes is applied to Lévy-type 
processes in order to obtain convergence results on discrete and continuous time indexed processes, simulation methods and Euler schemes.  We also apply the same theory to a slightly different situation, in order to get results of convergence of diffusions or random walks toward singular diffusions. As a consequence 
we deduce the convergence of random walks in random medium toward diffusions in 
random potential.}\\
{{\small\noindent{\bf Key words:}~ Lévy-type processes, random walks and diffusions in random and non-random environment, weak convergence of probability measures, discrete schemes, Skorokhod topology, martingale problem, Feller processes, generators}\\
{\small\noindent{\bf MSC2010 Subject Classification:}~Primary~60J25; Secondary~60J75, 60B10, 60G44,  60J35, 60J05, 60K37, 60E07, 47D07}

\section{Introduction}

Lévy-type processes constitute a large class of processes allowing to build models for 
many phenomena. Heuristically, a Lévy-type process is a Markov process taking its values 
in the one-point compactification $\R^{d\Delta}$, such that in each 
point of $a\in\R^d$,
\begin{itemize}
\item it is drifted by the value of a vector $\delta(a)$, 
\item it diffuse with the value $\gamma(a)$, a symmetric  positive semi-definite matrix,
\item it jumps into a Borel subset $B$ of $\R^{d\Delta}$ with the rate $\nu(a,B)$, a 
positive measure satisfying 
\[
\int(1\wedge |b-a|^2)\nu(a,db)<\infty.
\]
\end{itemize}
For a Lévy-type process we will call $(\delta,\gamma,\nu)$ its L\'evy triplet.

In general, the study of the convergence of sequences of general Markov processes is an important  question. The present paper consider this question, among others, in the setting of the preceding two models. The approximating Markov sequences could have continuous or discrete time 
parameter in order to cover scaling transformations or discrete schemes. 

A usual way to obtain such results is the use of the theory of Feller processes. In this 
context there exist two corresponding results of convergence (see, for instance Kallenberg
\cite{Ka02}, Thms. 19.25, p. 385 and 19.27, p. 387). However, on one hand, when one needs to consider unbounded coefficients, for instance the L\'evy triplet $(\delta,\gamma,\nu)$ for Lévy-type processes, 
technical difficulties could appear in the framework of Feller processes. On the other hand the cited 
results of convergence impose the knowledge of the generator. This is not the case 
in some constructions. 

Our method to tackle these difficulties is to consider the context of the martingale local problems and of locally Feller processes, introduced in \cite{GH117}. 
In this general framework we have already analysed the question of convergence of sequences of locally 
Feller processes. In the present paper we add the study of the convergence for processes 
indexed by a discrete time parameter toward processes indexed by a continuous time parameter.
We obtain the characterisation of the convergence in terms of convergence of associated 
operators, by using the uniform convergence on compact sets, and hence operators with 
unbounded coefficients  could be considered.
Likewise, we do not impose that the operator is a generator, but we assume only the well-posed feature of the associated martingale local 
problem. Indeed, it is more easy to verify the well-posed feature (see for instance, Stroock \cite{St75} for Lévy-type processes, Stroock and Varadhan \cite{SV06} for diffusion 
processes,  Kurtz \cite{Ku11} for Lévy-driven stochastic differential equations and forward equations...).

We apply our abstract results and we obtain sharp results of convergence for discrete and  continuous time sequences of processes toward Lévy-type process, in terms of Lévy parameters $(\delta,\gamma,\nu)$. We prefer the use of the Lévy triplet than the symbol 
associated to the operator, since the results are more precise in the situation of possibly instantaneous explosions. This is due essentially to the fact that the vague convergence of bounded measures cannot be characterised in terms of characteristic function. Our results can also be used to simulate Lévy-type processes and we improve Theorem 7.6 from B\"ottcher, Schilling and Wang \cite{BSW13}, p. 172, which is an approximation result of type  Euler scheme. 
We state the results in terms of convergence of operators, but essentially 
one can deduce the convergence of the associated processes.

Another well known model is the dynamic of a Brownian particle in a potential.  
It is often given by the solution of the one-dimensional stochastic differential equation
\[
\d X_t = \d B_t -\frac{1}{2}V^\prime(X_t)\d t,
\]
where $V:\R\to\R$. The process $X$ is also a L\'evy-type process, and, 
thanks to the regularising property of the Brownian motion one can consider very general potentials, for example \cadlag functions  (see Mandle \cite{Ma68}). In particular, it can be supposed that the potential is a Brownian path (see 
Brox \cite{Br86}), a Lévy paths (see Carmona\cite{Ca97}) or other random paths (Gaussian 
and/or fractional process ...).

When studying a Brownian particle in a potential, we prove the continuous dependence of 
the diffusion with respect to the potential, using again our  abstract results, even this is a different situation. We point 
out that it can be possible to consider potentials with very few constraints. In particular we 
consider diffusions in random potentials as limits of random walks in random mediums,
as an application of an approximation of the diffusion by random walks on $\Z$. An important example is the convergence of Sinai's random walk \cite{Si82} toward the diffusion in a Poisson potential (recovering Thm. 2 from Seignourel \cite{Se00}, p. 296),  toward the diffusion in a Brownian potential, also called Brox's diffusion (improving Thm. 1 from Seignourel \cite{Se00}, p. 295) and, more generally, toward the diffusion in a Lévy potential.

Let us describe the organisation of the paper. The next section contains
notations and results from  our previous paper \cite{GH117} which will be used in the 
present paper. In particular, we give the statements concerning  the existence of solutions 
for martingale local problems and concerning the convergence of continuous time locally 
Feller processes.  Section 3 is devoted to the limits of sequences of discrete time processes, while Section 4 contains two results of convergence toward general Lévy-type processes.
The diffusions evolving in a potential are studied in Section 5. The appendix contains the statements of several technical results already proved in \cite{GH117}.

\section{Martingale local problem setting and related results}

Let $S$ be a locally compact Polish space. 
Take $\Delta\not\in S$, and we will denote by $S^\Delta\supset S$ the one-point compactification of $S$, if $S$ is not compact, 
or the topological sum $S\sqcup\{\Delta\}$, if $S$ is compact  (so $\Delta$ is an isolated point). 
The fact that a subset $A$ is compactly embedded in an open subset $U\subset S^\Delta$ will be denoted by $A\Subset U$.
If  $x\in (S^\Delta)^{\R_+}$ we denote the explosion time by
\[
\xi(x):=\inf\{t\geq 0\,|\,\{x_s\}_{s\leq t}\not\Subset S\}.
\]



\noindent
The set of exploding \cadlag paths is defined by 
\begin{equation*}
\Dloc(S):=\left\{ x\in(S^\Delta)^{\R_+}\mid \begin{array}{l}
\forall t\geq\xi(x),~ x_t=\Delta,\\
\forall t\geq 0,~ x_t=\lim_{s\downarrow t}x_s,\\
\forall t>0\text{ s.t. }\{x_s\}_{s<t}\Subset S,~ x_{t-}:=\lim_{s\uparrow t}x_s\text{ exists}
\end{array}\right\},
\end{equation*}
and is endowed with the local Skorokhod topology (see Theorem 2.6 in\cite{GH017}) 
which is also Polish. 
A sequence $(x^k)_{k\in\N}$ in $\Dloc(S)$
converges to $x$ for the local Skorokhod topology if and only if there exists a sequence  $(\lambda^k)_k$ 
of increasing homeomorphisms on $\R_+$ satisfying 
\[
\forall t\geq 0\mbox{ s.t. }\{x_s\}_{s<t}\Subset S,\quad \lim_{k\to\infty}\sup_{s\leq t}d(x_s,x^k_{\lambda^k_s})=0
\quad\mbox{and}\quad
\lim_{k\to\infty}\sup_{s\leq t}|\lambda^k_s -s|=0.
\]
The local Skorokhod topology does not depend on the arbitrary  metric $d$ on $S^\Delta$, but only on the topology on $S$. 
We will always denote by $X$ the canonical process 
on $\Dloc(S)$.
We endow 
$\Dloc(S)$ with the 
Borel 
$\sigma$-algebra $\Fc:=\sigma(X_s,~0\leq s <\infty)$ and a filtration $\Fc_t:=\sigma(X_s,~0\leq s \leq t)$. 

Denote by $\rmC(S):=\rmC(S,\R)$, respectively by $\rmC(S^\Delta):=\rmC(S^\Delta,\R)$, the set of real continuous functions on $S$, respectively on $S^\Delta$, and by $\rmC_0(S)$ the set of functions $f\in\rmC(S)$ vanishing in $\Delta$.
We endow the set $\rmC(S)$ with the topology of uniform convergence on compact sets and  $\rmC_0(S)$ with the topology of uniform convergence.

We proceed by recalling the notion of  martingale local problem.  Let $L$ be a subset of $\rmC_0(S)\times\rmC(S)$. The set $\Mc(L)$ of solutions of the martingale local problem associated 
to $L$ is the set of probabilities $\Pbf\in\Pcal\left(\Dloc(S)\right)$ such that for all $(f,g)\in L$ and open subset $U\Subset S$:
\[
f(X_{t\wedge\tau^U})-\int_0^{t\wedge\tau^U}g(X_s)\d s\text{ is a }\Pbf\text{-martingale}
\]
with respect to the filtration $(\Fc_t)_t$ or, equivalent, to the filtration  $(\Fc_{t+})_t$.
Here $\tau^U$ is  the stopping time given by
\begin{equation}\label{eqtauU}
\tau^U:=\inf\left\{t\geq 0\mid X_t\not\in U\text{ or }X_{t-}\not\in U\right\}.
\end{equation}



In \cite{GH117} the following result of existence of solutions for martingale local problem was 
proved:
\begin{theorem}\label{thmExMP}
Let $L$ be a linear  subspace of $\rmC_0(S)\times\rmC(S)$ such that its domain  $D(L):=\left\{f\in\rmC_0(S)\mid\exists g\in\rmC(S),~(f,g)\in L\right\}$ is dense in $\rmC_0(S)$. Then, there is equivalence between 
\begin{enumerate}
\item[i)] existence of a solution for the martingale local problem: for any $a\in S$ there exists an element $\Pbf$ in $\Mc(L)$ such that $\Pbf(X_0=a)=1$;
\item[ii)] $L$ satisfies the positive maximum principle: for all $(f,g)\in L$ and $a_0\in S$, if $f(a_0)=\sup_{a\in S}f(a)\geq 0$ then $g(a_0)\leq 0$.
\end{enumerate}
\end{theorem}
\noindent
Let us note that a linear subspace $L\subset\rmC_0(S)\times\rmC(S)$ satisfying the positive maximum principle is univariate, so it can be equivalently considered as a linear operator
\[
L:\rmD(L)\to\rmC(S).
\]
The martingale local problem is said well-posed if there is existence and uniqueness of the solution, which means that 
for any $a\in S$ there exists an unique element $\Pbf$ in $\Mc(L)$ such that $\Pbf(X_0=a)=1$.


A family of probabilities  $(\Pbf_a)_a\in\Pcal(\Dloc(S))^S$ is called locally Feller if 
there exists $L\subset\rmC_0(S)\times\rmC(S)$ such that $\rmD(L)$ is dense in $\rmC_0(S)$ and
\[
\forall a\in S:\quad\quad
\Pbf\in\Mc(L)\,\mbox{ and }\,\Pbf(X_0=a)=1
\Longleftrightarrow
\Pbf=\Pbf_a.
\]
The $\rmC_0\times\rmC$-generator of a locally Feller family $(\Pbf_a)_a\in\Pcal(\Dloc(S))^S$ is the set of functions $(f,g)\in\rmC_0(S)\times\rmC(S)$ such that, for any $a\in S$ and any 
open subset $U\Subset S$,
\[
f(X_{t\wedge\tau^U})-\int_0^{t\wedge\tau^U}g(X_s)\d s\text{ is a }\Pbf_a\text{-martingale}.
\]
%

\noindent
It was noticed in Remark 4.11 in \cite{GH117} that if  $h\in\rmC(S,\R_+^*)$ and if
$L$ is the $\rmC_0\times\rmC$-generator of a locally Feller family, then 
\begin{equation}\label{notationhL}
hL:=\{(f,hg)\,|\,(f,g)\in L\}
\;\mbox{ is the $\rmC_0\times\rmC$-generator of a locally Feller family.}
\end{equation}

A family of probability measures associated to a Feller semi-group constitutes a natural example of locally Feller family (see Theorem 4.8 from \cite{GH117}).
We recall that a Feller semi-group $(T_t)_{t\in\R_+}$ is a strongly continuous semi-group   of positive linear contractions on $\rmC_0(S)$. 
Its  $\rmC_0\times\rmC_0$-generator is the set $L_0$ of $(f,g)\in\rmC_0(S)\times\rmC_0(S)$ such that, for all $a\in S$
\[
\lim_{t\to 0}\frac{1}{t}\big(T_tf(a)-f(a)\big)=g(a).
\]
It was proved in Propositions 4.2 and 4.12 from \cite{GH117}, that the martingale problem associated to 
$L_0$ admits a unique solution and, if $L$ denotes its  $\rmC_0(S)\times\rmC(S)$-generator  then, taking the closure in $\rmC_0(S)\times\rmC(S)$,
\begin{equation}\label{eqGenFF}
L_0=L\cap\rmC_0(S)\times\rmC_0(S)\quad\text{and}\quad L=\overline{L_0}.
\end{equation}
%

The following result of convergence is essential for our further development and it was proved in \cite{GH117}. As was already pointed out in the introduction, an improvement with respect to the classical result of convergence (for instance Theorem 19.25, p. 385, in \cite{Ka02}), is that one does not need to know the generator of the limit family, but only the fact that a martingale local problem is well-posed. 

\begin{theorem}[Convergence of locally Feller family]\label{thmCvgLocFel}
For $n\in\N\cup\{\infty\}$, let $(\Pbf^n_a)_a\in\Pcal(\Dloc(S))^S$ be a locally Feller family and let $L_n$ be a subset of $\rmC_0(S)\times\rmC(S)$. Suppose that for any $n\in\N$, $\overline{L_n}$ is the generator of $(\Pbf^n_a)_a$, suppose also that $\rmD(L_\infty)$ is dense in $\rmC_0(S)$ and
\[
\forall a\in S:\quad\quad
\Pbf\in\Mc(L_\infty)\,\mbox{ and }\,\Pbf(X_0=a)=1
\Longleftrightarrow
\Pbf=\Pbf^\infty_a.
\]
Then we have equivalence between:
\begin{enumerate}
\item[a)] the mapping
\[\begin{array}{ccc}
\N\cup\{\infty\}\times\Pcal(S^\Delta)&\to&\Pcal\left(\Dloc(S)\right)\\
(n,\mu)&\mapsto&\Pbf^n_\mu
\end{array}\]
is weakly continuous for the local Skorokhod topology, where $\Pbf_\mu:=\int\Pbf_a\mu(\d a)$ and $\Pbf_\Delta(X_0=\Delta)=1$;
\item[b)] for any $a_n,a\in S$ such that $a_n\to a$, 
$\Pbf^n_{a_n}$ converges weakly for the local Skorokhod topology to $\Pbf^\infty_a$, 
as $n\to\infty$; 
\item[c)] for any $f\in \rmD(L_\infty)$, there exist for each $n$, $f_n\in \rmD(L_n)$ such that $f_n\cv[\rmC_0]{n\to\infty} f$, $L_nf_n\cv[\rmC]{n\to\infty}L_\infty f$.
\end{enumerate}
\end{theorem}

\section{Convergence of  families indexed by discrete time}

We start our study by giving a discrete time version of the notion of locally Feller family.
\begin{definition}[Discrete time locally Feller family]\label{defDTSCMF}
We denote by $Y$ the discrete time canonical process on $(S^\Delta)^\N$ and we endow $(S^\Delta)^\N$ with the canonical $\sigma$-algebra.
A family $(\Pbf_a)_a\in\Pcal\left((S^\Delta)^\N\right)^S$ is said to be a discrete time locally Feller family if there exists an  operator $T:\rmC_0(S)\to\rmC_b(S)$, called  transition operator, 
such that
for any $a\in S$: $\Pbf_a(Y_0=a)=1$ and 
\begin{align}\label{eqDefDTSCMF1}
\forall n\in\N,\,\,\forall f\in\rmC_0(S),\quad
\Ebf_a\left(f(Y_{n+1})\mid Y_0,\ldots,Y_n\right)=\1_{\{Y_n\neq\Delta\}}Tf(Y_n)\quad\Pbf_a\text{-a.s.}
\end{align}
If we denote $\Pbf_\Delta$ the probability  defined by $\Pbf_\Delta(\forall n\in\N,~Y_n=\Delta)=1$, then for $\mu\in\Pcal(S^\Delta)$, $\Pbf_\mu:=\int\Pbf_a\mu(\d a)$
satisfies also \eqref{eqDefDTSCMF1}.
\end{definition}

Now we can state the main result of this section which, similarly, is an improvement with respect to Theorem 19.27, p. 387, in \cite{Ka02}, in the 
sense that one does not need to know the generator of the limit family, but only the fact that a martingale local problem is well-posed. 

\begin{theorem}[Convergence]\label{thmCvgLocFel2}
Let $L$ be a subset of $\rmC_0(S)\times\rmC(S)$ with $\rmD(L)$ a dense subset 
of  $\rmC_0(S)$, such that  the martingale local problem associated to $L$ is well-posed, and 
let $(\Pbf_a)_a\in\Pcal(\Dloc(S))^S$ be the associated continuous time locally Feller family. 
For each $n\in\N$ we introduce  $(\Pbf^n_a)_a\in\Pcal((S^\Delta)^\N)^S$ a
discrete time locally Feller families having their transition operator $T_n$. We denote 
by $L_n$ the operator $(T_n-{\rm id})/\eps_n$, where $(\eps_n)_n$ is a sequence of positive constants converging to $0$, as $n\to\infty$.
There is equivalence between:
\begin{enumerate}
\item[a)] for any $\mu_n,\mu\in \Pcal(S^\Delta)$ such that $\mu_n\to \mu$ weakly, as $n\to\infty$,
\[
\law_{\Pbf^n_{\mu_n}}\left((Y_{\lfloor t/\eps_n\rfloor})_t\right)\cv[\Pcal\left(\Dloc(S)\right)]{n\to\infty}\Pbf_\mu\,;
\]
\item[b)] for any $a_n,a\in S$ such that $a_n\to a$, as $n\to\infty$,
\[
\law_{\Pbf^n_{a_n}}\left((Y_{\lfloor t/\eps_n\rfloor})_t\right)\cv[\Pcal\left(\Dloc(S)\right)]{n\to\infty}\Pbf_a\,;
\]
\item[c)] for any $f\in \rmD(L)$, there exists, $(f_n)_n\in \rmC_0(S)^{\N}$ such that $f_n\cv[\rmC_0(S)]{n\to\infty}f$ and $L_nf_n\cv[\rmC(S)]{n\to\infty}Lf$.
\end{enumerate}
\end{theorem}

\noindent
Here $\lfloor r\rfloor$ denotes the integer part of the real number $r$.
\begin{proof}
Let $\Omega:=(S^\Delta)^\N\times\R_+^\N$ and $\Gc:=\Bc(S^\Delta)^{\otimes\N}\otimes\Bc(\R_+)^{\otimes\N}$ be. For any $\mu\in\Pcal(S^\Delta)$ and $n\in\N$, define $\P^n_\mu:=\Pbf^n_\mu\otimes\Ec(1)^{\otimes\N}$, where $\Ec(1)$ is the exponential distribution with expectation $1$.
Define 
\[\begin{array}{ccccc}
Y_n:&\Omega&\to & S&\quad\mbox{ and }\quad\\
&\left(((y_k)_k,(s_k)_k)\right)&\mapsto & y_n
\end{array} 
\begin{array}{cccc}
E_n:&\Omega&\to & \R_+\\
&\left(((y_k)_k,(s_k)_k)\right)&\mapsto & s_n,
\end{array}\]
and introduce the standard Poisson process
\[
\forall t\geq0,\quad N_t:=\inf\Big\{n\in\N\;\Big|\;\sum_{k=1}^{n+1}E_k>t\Big\}.
\]

\noindent
{\sl Step 1)} For each $n\in\N$ define $Z^n_t:=Y_{N_{t/\eps_n}}$. Consider the 
modified assertions:
\begin{enumerate}
\item[$a^{\prime})$] for any $\mu_n,\mu\in \Pcal(S^\Delta)$ such that $\mu_n\to \mu$,
\[
\law_{\P^n_{\mu_n}}\left(Z^n\right)\cv[\Pcal\left(\Dloc(S)\right)]{n\to\infty}\Pbf_\mu;
\]
\item[$b^{\prime})$] for any $a_n,a\in S$ such that $a_n\to a$,
\[
\law_{\P^n_{a_n}}\left(Z^n\right)\cv[\Pcal\left(\Dloc(S)\right)]{n\to\infty}\Pbf_a,
\]
\end{enumerate}
We will verify that $a^{\prime})\Leftrightarrow b^{\prime})\Leftrightarrow c)$. 
We need to prove that for all $\mu\in\Pcal(S^\Delta)$, $\law_{\P^n_\mu}(Z^n)\in\Mc(L_n)$. 
Taking $\Gc^n_t:=\sigma(N_{s/\eps_n},Z^n_s,~s\leq t)$, it is enough to prove 
that, for each $f\in\rmC_0(S)$ and $0\leq s\leq t$,
\[
\E^n_\mu\left[f(Z^n_t)-f(Z^n_s)-\int_s^tL_nf(Z_u^n)\d u\mid\Gc^n_s\right]=0.
\]
Let us introduce the $(\Gc^n_t)_t$-stopping times  $\tau^n_k:=\inf\Big\{u\geq 0\,\big|\, N_{u/\eps_n}=k\Big\}$. Then, for all $k\in\N$,
\begin{align*}
&\E^n_\mu\left[f(Z^n_{t\wedge(\tau^n_{k+1}\vee s)})-f(Z^n_{t\wedge(\tau^n_{k}\vee s)})\mid\Gc^n_{t\wedge(\tau^n_{k}\vee s)}\right]\\
&\quad=\1_{\{\substack{t>\tau^n_{k},s<\tau^n_{k+1}}\}}\E^n_\mu\left[(f(Y_{k+1})-f(Y_k))\1_{\{\tau^n_{k+1}\leq t\}}\mid\Gc^n_{t\wedge(\tau^n_{k}\vee s)}\right]\\
&\quad=\1_{\{\substack{t>\tau^n_{k},s<\tau^n_{k+1}}\}}
\E^n_\mu\left[(f(Y_{k+1})-f(Y_k))\1_{\{\tau^n_{k+1}-\tau^n_{k}\vee s\leq t-\tau^n_{k}\vee s\}}\mid\Gc^n_{\tau^n_{k}\vee s}\right]\\
&\quad=\1_{\{\substack{t>\tau^n_{k},s<\tau^n_{k+1}}\}}
(T_nf(Y_k)-f(Y_k))\big(1-\exp(-(t-\tau^n_{k}\vee s)/\eps_n)\big)\\
&\quad=\1_{\{\substack{t>\tau^n_{k},s<\tau^n_{k+1}}\}}
L_nf(Z^n_{\tau^n_{k}\vee s})\eps_n\big(1-\exp(-(t-\tau^n_{k}\vee s)/\eps_n)\big),
\end{align*}
where we used the fact that $(N_{u/\eps_n})_u$ is a Poisson process. Similarly, 
\begin{align*}
&\E^n_\mu\left[\int_{t\wedge(\tau^n_{k}\vee s)}^{t\wedge(\tau^n_{k+1}\vee s)}L_nf(Z_u^n)\d u\mid\Gc^n_{t\wedge(\tau^n_{k}\vee s)}\right]\\
&\quad=\1_{\{\substack{t>\tau^n_{k},s<\tau^n_{k+1}}\}}
L_nf(Z^n_{\tau^n_{k}\vee s})\E^n_\mu\left[t\wedge\tau^n_{k+1}-\tau^n_{k}\vee s\mid\Gc^n_{t\wedge(\tau^n_{k}\vee s)}\right]\\
&\quad=\1_{\{\substack{t>\tau^n_{k},s<\tau^n_{k+1}}\}}
L_nf(Z^n_{\tau^n_{k}\vee s})\E^n_\mu\left[(t-\tau^n_{k}\vee s)\wedge(\tau^n_{k+1}-\tau^n_{k}\vee s)\mid\Gc^n_{\tau^n_{k}\vee s}\right]\\
&\quad=\1_{\{\substack{t>\tau^n_{k},s<\tau^n_{k+1}}\}}
L_nf(Z^n_{\tau^n_{k}\vee s})\int_0^\infty (1/\eps_n)\exp(-u/\eps_n)
((t-\tau^n_{k}\vee s)\wedge u)\d u\\
&\quad=\1_{\{\substack{t>\tau^n_{k},s<\tau^n_{k+1}}\}}
L_nf(Z^n_{\tau^n_{k}\vee s})\eps_n\big(1-\exp(-(t-\tau^n_{k}\vee s)/\eps_n)\big).
\end{align*}
Hence
\[
\E^n_\mu\left[f(Z^n_{t\wedge(\tau^n_{k+1}\vee s)})-f(Z^n_{t\wedge(\tau^n_{k}\vee s)})-\int_{t\wedge(\tau^n_{k}\vee s)}^{t\wedge(\tau^n_{k+1}\vee s)}L_nf(Z_u^n)\d u\mid\Gc^n_{t\wedge(\tau^n_{k}\vee s)}\right]=0.
\]
Hence
\begin{align*}
&\E^n_\mu\left[f(Z^n_t)-f(Z^n_s)-\int_s^tL_nf(Z_u^n)\d u\mid\Gc^n_s\right]\\
& =\E^n_\mu\left[\sum_{k\geq 0}\left((Z^n_{t\wedge(\tau^n_{k+1}\vee s)})-f(Z^n_{t\wedge(\tau^n_{k}\vee s)})-\int_{t\wedge(\tau^n_{k}\vee s)}^{t\wedge(\tau^n_{k+1}\vee s)}L_nf(Z_u^n)\d u\right)\mid\Gc^n_s\right]\\
&=\sum_{k\geq 0}\E^n_\mu\left[\E^n_\mu\left[f(Z^n_{t\wedge(\tau^n_{k+1}\vee s)})-f(Z^n_{t\wedge(\tau^n_{k}\vee s)})-\int_{t\wedge(\tau^n_{k}\vee s)}^{t\wedge(\tau^n_{k+1}\vee s)}L_nf(Z_u^n)\d u\mid\Gc^n_{t\wedge(\tau^n_{k}\vee s)}\right]\mid\Gc^n_s\right]\\
&=0,
\end{align*}
so $\law_{\P^n_\mu}(Z^n)\in\Mc(L_n)$.
Hence applying the convergence theorem \ref{thmCvgLocFel} to $L_n$ and $L$, we obtain the equivalences between $a^\prime)$, 
$b^\prime)$ and $c)$. 

\noindent
{\sl Step 2.} To carry on with the proof we need the following technical result 
\begin{lemma}\label{lemCvgPbToLaw}
For $n\in\N$, let $(\Omega^n,\Gc^n,\P^n)$ be a probability space, let $Z^n:\Omega^n\to \Dloc(S)$  and $\Gamma^n:\Omega^n\to\rmC(\R_+,\R_+)$ be a increasing random  bijection. Define $\widetilde{Z}^n:=Z^n\circ \Gamma^n$.
Suppose that for each $\eps>0$ and $t\in\R_+$
\begin{align*}
\P^n\left(\sup_{s\leq t}|\Gamma^n_s-s|\geq\eps\right)\cv{n\to\infty}0.
\end{align*}
Then for any $\Pbf\in\Pcal(\Dloc(S))$,
\[
\law_{\P^n}(Z^n)\cv{n\to\infty}\Pbf\quad\Leftrightarrow\quad\law_{\P^n}(\widetilde{Z}^n)\cv{n\to\infty}\Pbf,
\]
where the limits are for the weak topology associated to the local Skorokhod topology.
\end{lemma}

We postpone the proof of this lemma and we finish the proof of the theorem. 
Let us note that for any $t\geq 0$ and $n\in\N$, $Y_{\lfloor t/\eps_n\rfloor}=Z^n_{\Gamma^n_t}$ with
\[
\Gamma^n_t:=\eps_n\left(\sum_{k=1}^{\lfloor t/\eps_n\rfloor}E_k+(t/\eps_n-\lfloor t/\eps_n\rfloor)E_{\lfloor t/\eps_n\rfloor+1}\right).
\]
Assuming that 
\begin{align}\label{eq1thmCvgLocFel2}
\forall t\geq 0,~\forall\eps>0,\quad\sup_{\mu\in \Pcal(S^\Delta)}\P^n_\mu
\left(\sup_{s\leq t}|\Gamma^n_s-s|\geq\eps\right)\cv{n\to\infty}0,
\end{align}
then, by the latter lemma we get $a)\Leftrightarrow a^\prime)$ and $b\Leftrightarrow b^\prime)$, so $a)\Leftrightarrow b)\Leftrightarrow c)$. 

Let us prove \eqref{eq1thmCvgLocFel2}. Fix  $t\geq 0$, $\eps >0$, $n\in\N$ and $\mu\in \Pcal(S^\Delta)$, then since $\Gamma^n$ is a continuous piecewise affine function, we have
\[
\sup_{s\leq t}|\Gamma^n_s-s|
\leq\sup_{\substack{k\in\N\\k\leq \lceil t/\eps_n\rceil}}|\Gamma^n_{k\eps_n}-k\eps_n|
=\sup_{\substack{k\in\N\\k\leq \lceil t/\eps_n\rceil}}\left|\eps_n\sum_{i=1}^kE_i-k\eps_n\right|
=\eps_n\sup_{\substack{k\in\N\\k\leq \lceil t/\eps_n\rceil}}|M_k|
\]
where  $M_k:=\sum_{i=1}^kE_i-k$  and $\lceil r\rceil$ denotes the smallest integer larger or equal than the real number $r$. Since the $E_i$ are independent random variables with exponential distribution $\Ec(1)$ we have
\[
\E^n_\mu[M_k^2]=k\E^n_\mu[(E_1-1)^2]=k.
\]
By Markov's inequality and by the maximal  Doob inequality applied to the discrete time martingale $(M_k)_k$ we can write
\begin{align*}
\P^n_\mu\left(\sup_{s\leq t}|\Gamma^n_s-s|\geq\eps\right)
&\leq\P^n_\mu\left(\eps_n\sup_{k\leq \lceil t/\eps_n\rceil}|M_k|\geq\eps\right)
\leq\frac{\E^n_\mu\left[\sup_{k\leq \lceil t/\eps_n\rceil}M_k^2\right]\eps_n^2}{\eps^2}\\
&\leq\frac{4\E^n_\mu\left[M_{\lceil t/\eps_n\rceil}^2\right]\eps_n^2}{\eps^2}
=\frac{4\lceil t/\eps_n\rceil\eps_n^2}{\eps^2}
\leq\frac{4(t+\eps_n)\eps_n}{\eps^2}.
\end{align*}
We deduce \eqref{eq1thmCvgLocFel2} and the proof of the theorem is complete except for 
the proof of Lemma \ref{lemCvgPbToLaw}.
\end{proof}

Before giving the proof of Lemma \ref{lemCvgPbToLaw} we state and prove a more general 
result:

\begin{lemma}\label{lmCvgPbToLaw}
Let $E$ be a Polish topological space, for $n\in\N$, let $(\Omega^n,\Gc^n,\P^n)$ be a probability space and consider  $Z^n,\widetilde{Z}^n:\Omega^n\to E$ random variables.
Suppose that for each compact subset $K\subset E$ and each open subset $U\subset E^2$ containing the diagonal $\{(z,z)\,|\,z\in E\}$,
\begin{align}\label{eq1LmCvgPbToLaw}
\P^n\left(Z^n\in K,~(Z^n,\widetilde{Z}^n)\not\in U\right)\cv{n\to\infty}0.
\end{align}
Then, for any $\Pbf\in\Pcal(E)$,
\[
\law_{\P^n}(Z^n)\cv{n\to\infty}\Pbf\quad\mbox{ implies }\quad\law_{\P^n}(\widetilde{Z}^n)\cv{n\to\infty}\Pbf,
\]
where the limits are for the weak topology on $\Pcal(E)$.
\end{lemma}
\begin{proof}
Suppose that $\law_{\P^n}(Z^n)\cv{n\to\infty}\Pbf$ so for any bounded continuous function $f:E\to \R$, $\lim_{n\to\infty}\E^n[f(Z^n)]=\int f\d\Pbf$.  Since $E$ is a Polish 
space the sequence $(\law_{\P^n}(Z^n))_n$ is tight. Take an arbitrary $\eps>0$ and let $K$ be a compact subset of $E$ such that
\begin{align}\label{eq2LmCvgPbToLaw}
\forall n\in\N,\quad\P^n(Z^n\not\in K)\leq\eps.
\end{align}
By \eqref{eq1LmCvgPbToLaw} applied to $K$ and $U:=\left\{(z,\widetilde{z})\mid |f(\widetilde{z})-f(z)|< \eps\right\}$, we have
\[
\P^n\left(Z^n\in K,~|f(\widetilde{Z}^n)-f(Z^n)|\geq\eps\right)\cv{n\to\infty}0.
\]
Hence by \eqref{eq2LmCvgPbToLaw}
\begin{align*}
\left|\E^n[f(\widetilde{Z}^n)]-\int f\d\Pbf\right| \hspace{-3cm}&\hspace{3cm}
\leq \left|\E^n[f(Z^n)]-\int f\d\Pbf\right|+\E^n\left|f(\widetilde{Z}^n)-f(Z^n)\right|\\
&\leq \left|\E^n[f(Z^n)]-\int f\d\Pbf\right|
+\E^n\left[\left|f(\widetilde{Z}^n)-f(Z^n)\right|\1_{\{Z^n\in K,|f(\widetilde{Z}^n)-f(Z^n)|\geq\eps\}}\right]\\
&\quad +\E^n\left[\left|f(\widetilde{Z}^n)-f(Z^n)\right|\1_{\{Z^n\in K,|f(\widetilde{Z}^n)-f(Z^n)|<\eps\}}\right]
+\E^n\left[\left|f(\widetilde{Z}^n)-f(Z^n)\right|\1_{\{Z^n\not\in K\}}\right]\\
&\leq \left|\E^n[f(Z^n)]-\int f\d\Pbf\right| +2\|f\|\,\P^n\left(Z^n\in K,~|f(\widetilde{Z}^n)-f(Z^n)|\geq\eps\right) +\eps(1+2\|f\|).
\end{align*}
Letting successively $n\to\infty$ and $\eps\to 0$, we deduce that
\[
\E^n[f(\widetilde{Z}^n)]\cv{n\to\infty}\int f\d\Pbf,
\]
hence, since $f$ is an arbitrary bounded continuous function, we have $\law_{\P^n}(\widetilde{Z}^n)\cv{n\to\infty}\Pbf$.
\end{proof}
\begin{proof}[Proof of Lemma  \ref{lemCvgPbToLaw}]
We denote by $\widetilde{\Lambda}$ the space of increasing bijections $\lambda$ from $\R_+$ to $\R_+$, and for $t\in\R_+$ we denote $\|\lambda-{\rm id}\|_t:=\sup_{s\leq t}|\lambda_s -s|$.
Since
\[
\forall\lambda\in\widetilde{\Lambda},~\forall t\in\R_+,~\forall\eps>0,\quad\|\lambda-{\rm id}\|_{t+\eps}<\eps\Rightarrow\|\lambda^{-1}-{\rm id}\|_t<\eps,
\]
the hypotheses of Lemma \ref{lemCvgPbToLaw} are symmetric with respect to $Z$ and $\widetilde{Z}$, so it suffices to prove only one implication. Hence we suppose $\law_{\P^n}(Z^n)\cv{n\to\infty}\Pbf$ and, by applying Lemma \ref{lmCvgPbToLaw}, we prove $\law_{\P^n}(\widetilde{Z}^n)\cv{n\to\infty}\Pbf$. Let $K$ be a compact subset of $\Dloc(S)$ and $U$ be an open subset of $\Dloc(S)^2$ containing the diagonal $\{(z,z)\,|\,z\in \Dloc(S)\}$. We prove the assertion
\begin{equation}\label{eq1CorCvgPbToLaw}
\exists t\geq 0,~\exists\eps>0,~\forall z\in K,~\forall \lambda\in\widetilde{\Lambda},\quad\|\lambda-{\rm id}\|_t<\eps\Rightarrow (z,z\circ\lambda)\in U.
\end{equation}
If we suppose that \eqref{eq1CorCvgPbToLaw} is false, then we can find two sequences $(z^n)_n\in K^\N$ and $(\lambda^n)_n\in\widetilde{\Lambda}^\N$ such that, for all $n\in\N$, $(z^n,z^n\circ\lambda^n)\not\in U$ and for all $t\geq 0$, $\|\lambda_n-{\rm id}\|_t\to 0$, as $n\to\infty$. By compactness of $K$, possibly by taking a subsequence, we may suppose the existence of $z\in K$ such that $z^n\to z$ as $n\to\infty$. Then, it is straightforward to obtain
\[
U\not\ni(z^n,z^n\circ\lambda^n)\cv{n\to\infty}(z,z)\in U.
\]
This is a contradiction with the fact that $U$ is open, so we have proved \eqref{eq1CorCvgPbToLaw}. Take $t$ and $\eps$ given by \eqref{eq1CorCvgPbToLaw}, then
\begin{align*}
\P^n\left(Z^n\in K,~(Z^n,\widetilde{Z}^n)\not\in U\right)\leq \P^n\left(\|\Gamma^n-{\rm id}\|_t\geq\eps\right)\cv{n\to\infty}0.
\end{align*}
Hence by Lemma \ref{lmCvgPbToLaw}, $\law_{\P^n}(\widetilde{Z}^n)\cv{n\to\infty}\Pbf$.
\end{proof}

\section{Lévy-type processes: convergence and discrete scheme}
In this section we take $d\in\N^*$, we denote by $|\cdot|$ the Euclidean norm and  by $\R^{d\Delta}$  the one point compactification of $\R^d$. Let also $\rmC^\infty_c(\R^d)$ be the set of compactly supported infinitely differentiable functions from $\R^d$ to $\R$. We are interested in the dynamics which locally looks like as Lévy's dynamics. Let us introduce a linear functional on 
$\rmC^\infty_c(\R^d)$ which describes a dynamic in a neighbourhood of a point $a\in\R^d$: for $f\in\rmC^\infty_c(\R^d)$,
\begin{align*}
T_{\chi,a}(\delta,\gamma,\nu)f:=\frac{1}{2}\sum_{i,j=1}^d\gamma_{ij}\partial^{2}_{ij}f(a)+\delta\cdot\nabla f(a)
+\int_{\R^{d\Delta}}(f(b)-f(a)-\chi(a,b)\cdot\nabla f(a))\nu(\d b),
\end{align*}
where
\begin{hypotheses}
\Hypo{H1}{\sl
-- the compensation function $\chi:\R^d\times\R^{d\Delta}\to\R^d$ is a bounded measurable function satisfying, for any compact subset $K\subset \R^d$,
\[
\sup_{{b,c\in K,\,b\not= c}}\frac{|\chi(b,c)-(c-b)|}{|c-b|^2}<\infty;
\]}
\Hypo{H2$(a)$}{\sl
-- the drift vector is $\delta\in\R^d$, the diffusion matrix $\gamma\in\R^{d\times d}$ is symmetric positive semi-definite and the jump measure $\nu$ is a measure on $\R^{d\Delta}$ satisfying $\nu(\{a\})=0$ and
\[
\int_{\R^{d\Delta}}(1\wedge |b-a|^2)\nu(db)<\infty.
\]}
\end{hypotheses}

\noindent
Usually we take for compensation function
\begin{equation}\label{explchi}
\chi_1(a,b):=(b-a)/(1+|b-a|^2)\quad\mbox{ or }\quad\chi_2(a,b):=(b-a)\1_{|b-a|<1}.
\end{equation}
It is well known (see for instance Theorem 2.12 pp. 21-22 from \cite{Ho98}) that for any linear operator $L:\rmC^\infty_c(\R^d)\to\rmC(\R^d)$ satisfying the positive maximum principle and for any $\chi$ satisfying (H1): for each $a\in\R^d$ there exist $\delta(a)$, $\gamma(a)$ and $\nu(a)$ satisfying (H2($a$)) such that
\begin{equation*}
\forall f\in\rmC^\infty_c(\R^d),~\forall a\in\R^d,\quad Lf(a)=T_{\chi,a}(\delta(a),\gamma(a),\nu(a))f.
\end{equation*}
In the following we will call a such expression of $L$ a {\sl Lévy-type operator}.

In order to obtain a converse sentence and to get the convergence of sequences of Lévy-type operators, 
we have to made a more restrictive hypothesis on the couple $(\chi,\nu)$: for $a\in\R^d$
\begin{hypotheses}
\Hypo{H3($a$)}{\sl
-- the compensation function $\chi:\R^d\times\R^{d\Delta}\to\R^d$ is a bounded measurable function satisfying,  for any compact subset $K\subset \R^d$,
\[
\sup_{{b,c\in K,\,0<|c-b|\leq\eps}}\frac{|\chi(b,c)-(c-b)|}{|c-b|^2}\cv{\eps\to0}0,
\]
and $\nu\left(\left\{b\in\R^{d\Delta}\mid\chi\text{ is not continuous at }(a,b)\right\}\right)=0$.
}
\end{hypotheses}

\noindent
For example,  $\chi_1$ given in \eqref{explchi} satisfies (H3($a$)) for any $\nu$ and  $\chi_2(a,b)$ satisfies (H3($a$)) whenever $\nu\big(\{b\in\R^d\,|\,|b-a|=1\}\big)=0$.

The following theorem contains a necessary and sufficient condition for the  convergence of sequences Lévy-type operators (and processes) in terms of their
L\'evy triplet. Before we introduce some notations.
\begin{itemize}
\item Let  $\chi:\R^d\times\R^{d\Delta}\to\R^d$ be a compensation function. For each $a\in\R^d$ let  $(\delta(a),\gamma(a),\nu(a))$ and $(\chi,\nu(a))$ be satisfying 
respectively {\rm (H2($a$))} and {\rm (H3($a$))}. Set
\begin{equation}\label{levytypeoperator}
Lf(a):=T_{\chi,a}(\delta(a),\gamma(a),\nu(a))f,\hspace{.7cm}\text{for any }f\in\rmC^\infty_c(\R^d).
\end{equation}
\item For each $n\in\N$ and $a\in\R^d$ let $(\delta_n(a),\gamma_n(a),\nu_n(a))$ be satisfying {\rm (H2($a$))}. Set
\begin{equation}
L_nf(a):=T_{\chi,a}(\delta_n(a),\gamma_n(a),\nu_n(a))f,\hspace{.7cm}\text{for any }f\in\rmC^\infty_c(\R^d).
\end{equation}
\item For each $n\in\N$ and $a\in\R^d$ let $\mu_n(a)$ be a probability measure on $\R^{d\Delta}$ and let $\eps_n>0$ be a sequence  converging to $0$. Set
\begin{equation}
T_nf(a):=\int f(b)\mu_n(a,\d b),\hspace{.7cm}\text{for any }f\in\rmC(\R^{d\Delta}).
\end{equation}
\end{itemize}
\begin{theorem}[Characterisation of the convergence toward Lévy-type operators]\label{thmCVGLT}\begin{samepage}~\\
1) The function $Lf$ is continuous for any $f\in\rmC^\infty_c(\R^d)$ if and only if
\begin{itemize}
\item $a\mapsto\delta(a)$ is continuous on $\R^d$,
\item $a\mapsto\int f(b)\nu(a,\d b)$ is continuous on the interior of $\{f=0\}\cap\R^d$, for any $f\in\rmC(\R^{d\Delta})$,
\item $a\mapsto\gamma_{ij}(a)+\int\chi_i(a,b)\chi_j(a,b)\nu(a,\d b)$ is continuous on $\R^d$, for any $1\leq i,j\leq d$.
\end{itemize}
\end{samepage}
\pagebreak[2]
\begin{samepage}
2) Assume that $Lf$ is continuous
for any $f\in\rmC^\infty_c(\R^d)$.  The uniform  convergence on compact sets, $L_nf\to Lf$, as $ n\to\infty$, holds for all $f\in\rmC^\infty_c(\R^d)$ if and only if
\begin{itemize}
\item $\delta_n(a)\to\delta(a)$, uniformly for $a$ varying in compact subsets of $\R^d$,
\item $\int f(b)\nu_n(a,\d b)\to\int f(b)\nu(a,\d b)$, uniformly for $a$ varying in compact subsets of the interior of $\{f=0\}\cap\R^d$, for any $f\in\rmC(\R^{d\Delta})$,
\item $\gamma_{n,ij}(a)+\int(\chi_i\chi_j)(a,b)\nu_n(a,\d b)\to\gamma_{ij}(a)+\int(\chi_i\chi_j)(a,b)\nu(a,\d b)$, uniformly for $a$ varying in compact subsets of $\R^d$, for any $1\leq i,j\leq d$.
\end{itemize}
\end{samepage}
\pagebreak[2]
\begin{samepage}
3) Assume that $Lf$ is continuous
for any $f\in\rmC^\infty_c(\R^d)$.
The uniform  convergence on compact sets,  $\eps_n^{-1}(T_nf-f)\to Lf$, as $n\to\infty$, holds for all $f\in\rmC^\infty_c(\R^d)$ if and only if
\begin{itemize}
\item $\eps_n^{-1}\int_{\R^{d\Delta}\setminus\{a\}} \chi(a,b)\mu_n(a,\d b)\to\delta(a)$, uniformly for $a$ varying in compact subsets of $\R^d$,
\item $\eps_n^{-1}\int f(b)\mu_n(a,\d b)\to\int f(b)\nu(a,\d b)$, uniformly for $a$ varying in the compact subsets of the interior of $\{f=0\}\cap\R^d$, for any $f\in\rmC(\R^{d\Delta})$,
\item $\eps_n^{-1}\int_{\R^{d\Delta}\setminus\{a\}}(\chi_i\chi_j)(a,b)\mu_n(a,\d b)\to\gamma_{ij}(a)+\int(\chi_i\chi_j)(a,b)\nu(a,\d b)$, uniformly for $a$ varying in the compact subsets of $\R^d$, for any $1\leq i,j\leq d$.
\end{itemize}
\end{samepage}
\end{theorem}
\begin{remark}
 Thanks to Theorems \ref{thmCvgLocFel} and \ref{thmCvgLocFel2} we can deduce from Theorem \ref{thmCVGLT} sharp results of convergence for the processes associated to $L_n$, $T_n$ and $L$. In particular, the third part of  Theorem \ref{thmCVGLT} one could be seen as an improvement of the classical Donsker theorem, and, for instance allows us to simulate Lévy-type processes. We illustrate this fact by the following example.
\end{remark}

\begin{example}[Symmetric stable type operator]
Let $c\in\rmC(\R^d,\R_+)$ and $\alpha\in\rmC(\R^d,(0,2))$ be and denote, for $f\in\rmC_0(\R^d)$ and $a\in\R^d$, 
\[
Lf(a):=\int_{\R^d}(f(b)-f(a)-(b-a)\mycdot\nabla f(a)\1_{|b-a|\leq 1})c(a)|b-a|^{-d-\alpha(a)}\d b.
\]
As a consequence of the first part of Theorem \ref{thmCVGLT}, $L$ maps $\rmC_0(\R^d)$ to $\rmC(\R^d)$. For $a\in\R^d$ and $n\in\N^*$, define the probability measure
\[
\mu_n(a,\d b):=\frac{c(a)}{n}|b-a|^{-d-\alpha(a)}\1_{|b-a|\geq\eps_n(a)}\d b,\quad\text{ with }\quad\eps_n(a):=\left(\frac{c(a)S_{d-1}}{n\alpha(a)}\right)^{1/\alpha(a)}
\]
and where
$
S_{d-1}=2\pi^{d/2}/\Gamma(d/2)
$
is the measure of the unity sphere in $\R^d$.
Thanks to the third part of Theorem \ref{thmCVGLT}, for any $f\in\rmC^\infty_c(\R^d)$,
\[
\lim_{n\to\infty}n\left(\int f(b)\mu_n(a,\d b)-f(a)\right)=Lf(a),\,\mbox{ uniformly for $a$ in compact subsets of $\R^d$}.
\]
To go further, it is straightforward that for any $a\in\R^d$ and $n\in\N^*$, $\mu_n(a)$ is the distribution of the random variable
\[
a+Q\left(\frac{c(a)S_{d-1}}{n\alpha(a)U}\right)^{1/\alpha(a)},\;\text{ with independent }\; Q\sim\Uc(\mathbb{S}^{d-1}),~U\sim\Uc([0,1]),
\]
where $\Uc(\mathbb{S}^{d-1})$ and $\Uc([0,1])$ are the uniform distributions, respectively on the unity sphere of $\R^d$ and on $[0,1]$. To simulate the discrete time locally Feller processes associated to $(\mu_n(a))_a$ we can proceed as follows. Let $(Q_k,U_k)_k$ be an i.i.d. sequence of  random variables with distributions $\Uc(\mathbb{S}^{d-1})\otimes\Uc([0,1])$ and define, for $n\in\N^*$ and $k\in\N$,
\[
Z^n_{k+1}:=Z^n_k+Q_k\left(\frac{c(Z^n_k)S_{d-1}}{n\alpha(Z^n_k)U_k}\right)^{1/\alpha(Z^n_k)}.
\]
Hence thanks to Theorem \ref{thmCvgLocFel2}, if the martingale local problem associated to $L$ is well-posed, then $(Z^n_{\lfloor nt\rfloor})_t$ converges in distribution to the solution of the martingale local problem.
\end{example}
\begin{remark}
This example is adaptable when we want to simulate more general Lévy-type processes. The heuristics is as follows: first we approximate the Lévy measure by finite measures, we renormalise them and then we convolute with a Gaussian measure having  
well chosen parameters.
\end{remark}

Before proceeding to the proof Theorem \ref{thmCVGLT}, we give a second approximation result inspired 
from  \cite{BSW13}, Theorem 7.6 p. 172.
Let $L:\rmC_c^\infty(\R^d)\to\rmC(\R^d)$ be an operator satisfying the positive 
maximum principle. Let the translation of $f$ by $h\in\R^d$ be the mapping $\tau_hf(a)=f(a+h)$. For $a_0\in\R^d$, we introduce the operator 
\begin{equation}\label{eqLevOp}
L(a_0):\rmC^\infty_c(\R^d)\to\rmC_0(\R^d)\quad\mbox{ by }\quad L(a_0)f(a):=L(\tau_{a-a_{0}}f)(a_0).
\end{equation}
Clearly $Lf(a)=L(a)f(a)$.
Since $L(a_0)$ is invariant with respect  to the translation and satisfies the positive 
maximum principle then its closure  in $\rmC_0(\R^d)\times\rmC_0(\R^d)$ is the $\rmC_0\times\rmC_0$-generator of a Lévy family (see for instance, Section 2.1 pp. 32-41 from \cite{BSW13}). We denote by $(T_t(a_0))_{t\geq 0}$ its Feller semi-group and we state:
\begin{theorem}[Approximation with Lévy increments]\label{thmAppLT2}~\\
Let $(\eps_n)_n$ be a sequence of positive numbers such that $\eps_n\to 0$ and define the transition operators $T_n$ by
\[
T_nf(a):=T_{\eps_n}(a)f(a),\quad\mbox{ for }\;f\in\rmC_0(\R^d).
\]
Then, for any $f\in\rmC^\infty_c(\R^d)$,
\[
\frac{1}{\eps_n}(T_nf-f)\cv{n\to\infty}Lf,\quad\mbox{ uniformly on compact sets.}
\]
\end{theorem}
\begin{remark}\label{rkAppLT2}
If the martingale local problem associated to $L$ is well-posed, by  Theorem \ref{thmCvgLocFel2}, one deduces the convergence of the associated probability families.
\end{remark}
Excepting the fact that the present convergence is for the local Skorokhod topology, Theorem \ref{thmAppLT2} is an improvement of Theorem 7.6 p. 172 from \cite{BSW13}. More precisely,  
we do not need  that the closure of $L$ is a generator of a Feller semi-group, 
but we only suppose that the martingale local problem is well-posed.
We postpone the proof of the latter theorem to the end of this section.
\medskip

The proof of Theorem \ref{thmCVGLT} is obtained as a straightforward application of the following:
\begin{proposition}\label{propCVGLT}
For each $n\in\N\cup\{\infty\}$ take $a_n\in\R^d$ such that $a_n\to a_\infty$ and consider $(\delta_n,\gamma_n,\nu_n)$ satisfying {\rm(H2($a_n$))}. Let also $\chi$ be such that the couple $(\chi,\nu_\infty)$ satisfies {\rm(H3($a_\infty$))}.
Then
\begin{align}\label{eqPropCVGLT}
\forall f\in\rmC^\infty_c(\R^d),\quad T_{\chi,a_n}(\delta_n,\gamma_n,\nu_n)f\cv{n\to\infty} T_{\chi,a_{\infty}}(\delta_{\infty},\gamma_{\infty},\nu_{\infty})f,
\end{align}
if and only if the following three conditions hold
\begin{align}\label{equivPropCVGLT}
\left\{\begin{array}{l}
\delta_n\cv{n\to\infty}\delta_\infty,\\\\
\forall f\in\rmC(\R^{d\Delta})\mbox{ vanishing in a neighbourhood of $a_\infty$,}
\int f(b)\nu_n(\d b)\cv{n\to\infty}\int f(b)\nu_\infty(\d b),\\\\
\Big(\gamma_{n,ij}+\int(\chi_i\chi_j)(a_n,b)\nu_n(\d b)\Big)_{i,j}\,\cv{n\to\infty}\Big(\gamma_{\infty,ij}+\int(\chi_i\chi_j)(a_\infty,b)\nu_\infty(\d b)\Big)_{i,j}\,.
\end{array}\right.
\end{align}
\end{proposition}
\begin{proof}[Proof of Theorem \ref{thmCVGLT}]
Parts 1) and 2) are direct consequences of the latter proposition. To verify Part 3) 
we remark that for any $f\in\rmC^\infty_c(\R^d)$, $n\in\N$ and $a\in\R^d$, we have \[(T_nf(a)-f(a))/\eps_n=T_{\chi,a}(\delta_n(a),0,\nu_n(a))f\] with  \[\delta_n(a):=\eps_n^{-1}\int_{\R^{d\Delta}\setminus\{a\}} \chi(a,b)\mu_n(a,\d b)
\quad\mbox{ and }\quad \nu_n(a,\d b):=\eps_n^{-1}\1_{\R^{d\Delta}\setminus\{a\}}(b)\mu_n(a,\d b),\] hence we can apply again Proposition \ref{propCVGLT}.
\end{proof}
In order, to prove Proposition \ref{propCVGLT} we need the following lemma on the convergence of measures:
\begin{lemma}\label{lmCvMes}
For $n\in\N\cup\{\infty\}$ let $a_n\in\R^d$ be such that $a_n\to a_\infty$ 
and let $\nu_n$ be Radon measures on $\R^{d\Delta}\setminus\{a_n\}$.
 Suppose that, for any $f\in\rmC(\R^{d\Delta})$ such that $f$ vanishes in a neighbourhood of $a_\infty$, is constant in a neighbourhood of $\Delta$ and is infinitely differentiable in $\R^d$, we have
\[\int f(b)\nu_n(\d b)\cv{n\to\infty}\int f(b)\nu_\infty(\d b).\]
i) Then, for any sequence $(f_n)_{n\in\N\cup\{\infty\}}$ of measurable uniformly bounded functions from $\R^{d\Delta}$ to $\R$ such that the $f_n$ vanish in the same neighbourhood of $a_\infty$, for $n\in\N\cup\{\infty\}$, and such that
\begin{align}\label{eqLmCvMes}
\nu_\infty\Big(\R^{d\Delta}\setminus\big\{b_0\in\R^{d\Delta}\,\big|\, \lim_{n\to\infty,b\to b_0}
f_n(b)
=
f_\infty(b_0)\big\}\Big)=0,
\end{align}
we have 
\[\int f_n(b)\nu_n(\d b)\cv{n\to\infty}\int f_\infty(b)\nu_\infty(\d b).\]
ii) Assume, moreover, that there exists $\eta>0$ such that \[\sup_{n\in\N\cup\{\infty\}}\int|b-a_n|^2\1_{|b-a_n|\leq\eta}\,\nu_n(\d b)<\infty.\]
Then, for any sequence $(f_n)_{n\in\N\cup\{\infty\}}$ of measurable uniformly bounded functions from $\R^{d\Delta}$ to $\R$ satisfying $f_n(a_n)=0$, 
\begin{equation}\label{eqo}
\lim_{\delta\to 0}\limsup_{n\to\infty}\sup_{0<|h|\leq\delta}\frac{f_n(a_n+h)}{|h|^2}=0,
\end{equation}
and \eqref{eqLmCvMes}, we have
\[\int f_n(b)\nu_n(\d b)\cv{n\to\infty}\int f_\infty(b)\nu_\infty(\d b).\]
\end{lemma}
\begin{proof}
Consider a sequence of functions $(f_n)_{n\in\N\cup\{\infty\}}$ as in the first part of 
lemma.
Let  $U_1$ be an open subset such that $U_1\Subset\R^{d\Delta}\setminus\{a_\infty\}$ and 
\[
U_1\supset\bigcup_{n\in\N\cup\{\infty\}}\{f_n\not=0\}
\supset\R^{d\Delta}\setminus\Big\{b_0\in\R^{d\Delta}\,\Big|\,
\lim_{n\to\infty, b\to b_0}f_n(b)
=f_\infty(b_0)\Big\}.
\]
Let $\varphi_1\in\rmC(\R^{d\Delta})$ be such that $\varphi_1$ vanishes in a neighbourhood of $a_\infty$ and is constant in a neighbourhood of $\Delta$, $\varphi_1$ is infinitely differentiable in $\R^d$ and such that $\varphi_1\geq\1_{U_1}$. Then 
\[\int \varphi_1(b)\nu_n(\d b)\cv{n\to\infty}\int \varphi_1(b)\nu_\infty(\d b),\]
so
\[
\sup_{n\in\N\cup\{\infty\}}\nu_n(U_1)\leq\sup_{n\in\N\cup\{\infty\}}\int \varphi_1(b)\nu_n(\d b)<\infty.
\]
Since $\R^{d\Delta}\setminus\{a_\infty\}$ is a Polish space, the measure $\nu_\infty$ is inner regular on this set. Hence, if $\eps>0$ is chosen arbitrary,  there exists a compact subset $K_\eps\subset U_1$ satisfying
\begin{equation}\label{eqKeps}
K_\eps\subset
\Big\{b_0\in\R^{d\Delta}\,\big|\,
\lim_{n\to\infty, b\to b_0}f_n(b)
=f_\infty(b_0)\Big\}
\quad\text{and}\quad\nu_\infty(K_\eps)\geq\nu_\infty(U_1)-\eps.
\end{equation}
Hence $f_\infty$ is continuous on $K_\eps$ and $f_n$ converges uniformly to $f_\infty$ on $K_\eps$. There exists a function $\varphi_2\in\rmC(\R^{d\Delta})$ such that  $\varphi_2$ is constant in a neighbourhood of $\Delta$,  is infinitely differentiable in $\R^d$ and such that $\{\varphi_2\not=0\}\subset U_1$, $\|\varphi_2\|\leq\|f_\infty\|$ and $\|\varphi_2-f_\infty\|_{K_\eps}\leq \eps$.
Since \eqref{eqKeps}, by compactness there exists an open subset $U_2\subset U_1$  such that
\[
K_\eps\subset U_2\subset\Big\{b_0\in\R^{d\Delta}\,\big|\,\limsup_{n\to\infty,b\to b_0}\big|f_n(b)-\varphi_2(b_0)\big|\leq 2\eps\Big\}.
\]
By dominated convergence there exists a function $\varphi_3\in\rmC(\R^{d\Delta})$ such that $\varphi_3$ vanishes in a neighbourhood of $a_\infty$,  is constant in a neighbourhood of $\Delta$,  is infinitely differentiable in $\R^d$, and such that $\1_{U_2}\geq\varphi_3$ and $\int \varphi_3(b)\nu_\infty(\d b)\geq\nu_\infty(U_2)-\eps$. Hence
\begin{align*}
\liminf_{n\to\infty}\nu_n(U_2)
&\geq\liminf_{n\to\infty}\int \varphi_3(b)\nu_n(\d b)=\int \varphi_3(b)\nu_\infty(\d b)\geq\nu_\infty(U_2)-\eps\\
&\geq\nu_\infty(K_\eps)-\eps\geq\nu_\infty(U_1)-2\eps.
\end{align*}
Therefore  we have
\begin{align*}
\limsup_{n\to\infty}\left|\int f_n(b)\nu_n(\d b)-\int f_\infty(b)\nu_\infty(\d b)\right|
&\leq \limsup_{n\to\infty}\left|\int \varphi_2(b)\nu_n(\d b)-\int \varphi_2(b)\nu_\infty(\d b)\right|\\
&\hspace{-6cm}\quad+\limsup_{n\to\infty}\left|\int_{U_2} (f_n(b)-\varphi_2(b))\nu_n(\d b)\right|
+\limsup_{n\to\infty}\left|\int_{U_1\setminus U_2} (f_n(b)-\varphi_2(b))\nu_n(\d b)\right|\\
&\hspace{-6cm}\quad+\limsup_{n\to\infty}\left|\int_{K_\eps} (f_\infty(b)-\varphi_2(b))\nu_\infty(\d b)\right|
+\limsup_{n\to\infty}\left|\int_{U_1\setminus K_\eps} (f_\infty(b)-\varphi_2(b))\nu_\infty(\d b)\right|\\
&\hspace{-6cm}\leq 0+2\eps\sup_{n\in\N}\nu_n(U_1)+4\eps\sup_{n\in\N\cup\{\infty\}}\|f_n\|+\eps\nu_\infty(U_1)+2\eps\|f_\infty\|\\
&\hspace{-6cm}\leq 3\eps(\sup_{n\in\N\cup\{\infty\}}\nu_n(U_1)+2\sup_{n\in\N\cup\{\infty\}}\|f_n\|).
\end{align*}
Letting $\eps\to 0$ we obtain that 
\[\int f_n(b)\nu_n(\d b)\cv{n\to\infty}\int f_\infty(b)\nu_\infty(\d b).\]

We proceed with the proof of the part {\sl ii)} of lemma. 
Fix $\eta>0$ as in the statement and choose an arbitrary $\eps>0$. By 
\eqref{eqo}, there exists $0<\delta<\eta/2$ such that 
\[
\limsup_{n\to\infty}\sup_{0<|h|\leq2\delta}\frac{f_n(a_n+h)}{|h|^2}
\leq
\frac{\eps}{\displaystyle1\vee\sup_{n\in\N\cup\{\infty\}}\int|b-a_n|^2\1_{|b-a_n|
\leq\eta}\nu_n(\d b)}.
\]
Consider a function $\varphi\in\rmC(\R^{d\Delta},[0,1])$ which vanishes in a neighbourhood of $a_\infty$ and such that $\varphi(a)=1$ for any $a$ satisfying $|a-a_\infty|\geq\delta$.  Then, by the first part {\sl i)},
\[\int \varphi(b)f_n(b)\nu_n(\d b)\cv{n\to\infty}\int \varphi(b)f_\infty(b)\nu_\infty(\d b).\] 
For $n\in\N$ large enough,  $|a-a_n|\leq\delta$, hence
\[
\left|\int(1-\varphi(b))f_n(b)\nu_n(\d b)\right|\leq \int|b-a_n|^2\1_{|b-a_n|\leq\eta}\nu_n(\d b)\sup_{0<|h|\leq2\delta}\frac{f_n(a_n+h)}{|h|^2}\,,
\]
so $\displaystyle\limsup_{n\to\infty}\left|\int(1-\varphi(b))f_n(b)\nu_n(\d b)\right|\leq\eps$. We also have
\begin{multline*}
\left|\int(1-\varphi(b))f_\infty(b)\nu_\infty(\d b)\right|\\\leq 
\int|b-a_\infty|^2\1_{|b-a_\infty|\leq\eta}\nu_n(\d b)
\limsup_{n\to\infty}\sup_{0<|h|\leq2\delta}\frac{f_n(a_n+h)}{|h|^2}
\leq \eps,
\end{multline*}
so 
\[\limsup_{n\to\infty}
\left|\int f_n(b)\nu_n(\d b)-\int f_\infty(b)\nu_\infty(\d b)\right|\leq2\eps.
\] 
Letting $\eps\to 0$ we can conclude.
\end{proof}
\begin{proof}[Proof of Proposition \ref{propCVGLT}]
Suppose first \eqref{eqPropCVGLT}.  
Let $f\in\rmC(\R^{d\Delta})$ be such that $f$ vanishes in a neighbourhood of $a_\infty$, is constant in a neighbourhood of $\Delta$ and is infinitely differentiable in $\R^d$. Hence $f-f(\Delta)\in\rmC^\infty_c(\R^d)$, and
 \[T_{\chi,a_\infty}(\delta_\infty,\gamma_\infty,\nu_\infty)(f-f(\Delta))=\int f(b)\nu_\infty(\d b),
\]
while, for $n$ large enough,
\[
 T_{\chi,a_n}(\delta_n,\gamma_n,\nu_n)(f-f(\Delta))= \int f(b)\nu_n(\d b).\]
We deduce that 
\[\int f(b)\nu_n(\d b)\cv{n\to\infty}\int f(b)\nu_\infty(\d b).\]
Therefore we can apply the first part of Lemma \ref{lmCvMes} and in particular, 
for any $f\in\rmC(\R^{d\Delta})$ vanishing in a neighbourhood of $a$,
\[\int f(b)\nu_n(\d b)\cv{n\to\infty}\int f(b)\nu_\infty(\d b).\] 

Consider $\theta\in\rmC(\R_+,[0,1])$ such that $\theta(r)=1$ for $r\leq 1$ and $\theta(r)=0$ for $r\geq 2$.  For $(a,b)\in \R^d\times\R^{d\Delta}$ and $n\in\N\cup\{\infty\}$, define 
\begin{equation}\label{eqPrPropCvgLT1}
\widetilde{\chi}(a,b):=\theta(|b-a|)(b-a)\1_{b\not =\Delta}\;\mbox{ and }
\;
\widetilde{\delta}_n:=\delta_n+\int(\widetilde{\chi}(a_n,b)-\chi(a_n,b))\nu_n(\d b).
\end{equation}
Therefore, for all $f\in\rmC^\infty_c(\R^d)$ and all $n\in\N\cup\{\infty\}$
\[
T_{\chi,a_n}(\delta_n,\gamma_n,\nu_n)f=T_{\;{\widetilde\chi},a_n}(\widetilde{\delta}_n,\gamma_n,\nu_n)f.
\]
Let $\phi$ be an arbitrary linear form on $\R^d$ and consider $f\in\rmC^\infty_c(\R^d)$ such that $f(b)=\phi\cdot(b-a_\infty)$ in a neighbourhood of $a_\infty$. Then
\[
T_{\;\widetilde{\chi},a_\infty}(\widetilde{\delta}_\infty,\gamma_\infty,\nu_\infty)f=\phi\cdot\widetilde{\delta}_\infty+\int(f(b)-\phi\cdot\widetilde{\chi}(a_\infty,b))\nu_\infty(\d b)
\]
and for $n$ large enough
\[
T_{\;\widetilde{\chi},a_n}(\widetilde{\delta}_n,\gamma_n,\nu_n)f=\phi\cdot\widetilde{\delta}_n+\int(f(b)-f(a_n)-\phi\cdot\widetilde{\chi}(a_n,b))\nu_n(\d b).
\]
Thanks to the first part of Lemma \ref{lmCvMes}
\[
\int(f(b)-f(a_n)-\phi\cdot\widetilde{\chi}(a_n,b))\nu_n(\d b)
\cv{n\to\infty}
\int(f(b)-\phi\cdot\widetilde{\chi}(a_\infty,b))\nu_\infty(\d b),
\]
so $\phi\cdot\widetilde{\delta}_n\cv{n\to\infty}\phi\cdot\widetilde{\delta}_\infty$ and since $\phi$ was chosen arbitrary, $\widetilde{\delta}_n\cv{n\to\infty}\widetilde{\delta}_\infty$.

Let $\Phi$ be an arbitrary symmetric bilinear form on $\R^d$ and if $(e_1,\ldots,e_d)$ is the canonical basis 
of $\R^d$, we denote $\Phi_{ij}=\Phi(e_i,e_j)$, $i,j=1,\ldots,d$. Consider $f\in\rmC^\infty_c(\R^d)$ such that $f(b)=\Phi(b-a_\infty,b-a_\infty)$ in a neighbourhood of $a_\infty$. Then, for $n$ large enough,
\begin{align*}
&T_{\widetilde{\chi},a_n}(\widetilde{\delta}_n,\gamma_n,\nu_n)f\\
&=\sum_{i,j=1}^d\Phi_{ij}\gamma_{n,ij}+2\Phi(a_n-a_\infty,\widetilde{\delta}_n)
+\int(f(b)-f(a_n)-2\Phi(a_n-a_\infty,\widetilde{\chi}(a_n,b))\nu_n(\d b)\\
&=\sum_{i,j=1}^d\Phi_{ij}\left(\gamma_{n,ij}+\int(\widetilde{\chi}_i\widetilde{\chi}_j)(a_n,b)\nu_n(\d b))\right)+2\Phi(a_n-a_\infty,\widetilde{\delta}_n)\\
&\quad+\int\Big(f(b)-f(a_n)-2\Phi(a_n-a_\infty,\widetilde{\chi}(a_n,b))-\sum_{i,j=1}^d\Phi_{ij}\,(\widetilde{\chi}_i\widetilde{\chi}_j)(a_n,b)\Big)\nu_n(\d b).
\end{align*}
A similar equality holds with the index $n$ replaced by $\infty$:
\begin{multline*}
T_{\widetilde{\chi},a_\infty}(\widetilde{\delta}_\infty,\gamma_\infty,\nu_\infty)f
=\sum_{i,j=1}^d\Phi_{ij}\gamma_{\infty,ij}
+\int f(b)\nu_\infty(\d b)\\
=\sum_{i,j=1}^d\Phi_{ij}\left(\gamma_{\infty,ij}
+\int(\widetilde{\chi}_i\widetilde{\chi}_j)(a_\infty,b)\nu_\infty(\d b))\right)
+\int\Big(f(b)
-\sum_{i,j=1}^d\Phi_{ij}\,(\widetilde{\chi}_i\widetilde{\chi}_j)(a_\infty,b)\Big)
\nu_\infty(\d b).
\end{multline*}
Thanks to the first part of Lemma \ref{lmCvMes}
\begin{multline*}
\int\Big(f(b)-f(a_n)-2\Phi(a_n-a_\infty,\widetilde{\chi}(a_n,b))-\sum_{i,j=1}^d\Phi_{ij}\,(\widetilde{\chi}_i\widetilde{\chi}_j)(a_n,b)\Big)\nu_n(\d b)\\
\cv{n\to\infty}
\int\Big(f(b)-\sum_{i,j=1}^d\Phi_{ij}\,(\widetilde{\chi}_i\widetilde{\chi}_j)(a_\infty,b)\Big)
\nu_\infty(\d b),
\end{multline*}
so
\begin{equation*}
\sum_{i,j=1}^d\Phi_{ij}
\Big(\gamma_{n,ij}+\int(\widetilde{\chi}_i\widetilde{\chi}_j)(a_n,b)
\nu_n(\d b))\Big)
\cv{n\to\infty}
\sum_{i,j=1}^d\Phi_{ij}
\Big(\gamma_{\infty,ij}+\int(\widetilde{\chi}_i\widetilde{\chi}_j)(a_\infty,b)
\nu_\infty(\d b))\Big).
\end{equation*}
Since $\Phi$ was chosen arbitrary, for all $1\leq i,j\leq d$
\[
\gamma_{n,ij}+
\int(\widetilde{\chi}_i\widetilde{\chi}_j)(a_n,b)\nu_n(\d b))
\cv{n\to\infty}
\gamma_{\infty,ij}+
\int(\widetilde{\chi}_i\widetilde{\chi}_j)(a_\infty,b)\nu_\infty(\d b)).
\]
So we can apply the second part of Lemma \ref{lmCvMes} and in particular
\[\lim_{n\to\infty}\int(\widetilde{\chi}(a_n,b)-\chi(a_n,b))\nu_n(\d b)
=\int(\widetilde{\chi}(a_\infty,b)-\chi(a_\infty,b))\nu_\infty(\d b),\]
so by \eqref{eqPrPropCvgLT1}, $\delta_n\cv{n\to\infty}\delta_\infty$.
By the second part of Lemma \ref{lmCvMes} we also have, for all $1\leq i,j\leq d$,
\begin{equation*}
\int((\widetilde{\chi}_i\widetilde{\chi}_j)(a_n,b)-(\chi_i\chi_j)(a_n,b))\nu_n(\d b)
\cv{n\to\infty}
\int((\widetilde{\chi}_i\widetilde{\chi}_j)(a_\infty,b)-(\chi_i\chi_j)(a_\infty,b))\nu_\infty(\d b),
\end{equation*}
so we deduce
\[
\gamma_{n,ij}+\int(\chi_i\chi_j)(a_n,b)\nu_n(\d b)
\cv{n\to\infty}
\gamma_{\infty,ij}+\int(\chi_i\chi_j)(a_\infty,b)\nu_\infty(\d b).
\]
We prove the converse, by supposing \eqref{equivPropCVGLT} and  applying the second part of Lemma \ref{lmCvMes}. Let $f\in\rmC^\infty_c(\R^d)$ be, for all $n\in\N\cup\{\infty\}$,
\begin{multline*}\begin{aligned}
T_{\chi,a_n}(\delta_n,\gamma_n,\nu_n)f
:=&\frac{1}{2}\sum_{i,j=1}^d\gamma_{n,ij}\partial^{2}_{ij}f(a_n)+\delta_n\cdot\nabla f(a_n)\\
&+\int(f(b)-f(a_n)-\chi(a_n,b)\cdot\nabla f(a_n))\nu_n(\d b)
\end{aligned}\\\begin{aligned}
=&\frac{1}{2}\sum_{i,j=1}^d\left(\gamma_{n,ij}+\int(\chi_i\chi_j)(a_n,b)\nu_n(\d b)\right)\partial^2_{ij}f(a_n)+\delta_n\cdot\nabla f(a_n)\\
&+\int\left(f(b)-f(a_n)-\chi(a_n,b)\cdot\nabla f(a_n)-\sum_{i,j=1}^d(\chi_i\chi_j)(a_n,b)\partial^{2}_{ij}f(a_n)\right)\nu_n(\d b).
\end{aligned}\end{multline*}
Applying the second part of Lemma \ref{lmCvMes} to the last term of the previous equation we deduce
\[
T_{\chi,a_n}(\delta_n,\gamma_n,\nu_n)f
\cv{n\to\infty}
T_{\chi,a_\infty}(\delta_\infty,\gamma_\infty,\nu_\infty)f.\qedhere
\]
\end{proof}

We finish the section with the proof of Theorem \ref{thmAppLT2}.
Recall that $\chi_1(a,b)$ is given by \eqref{explchi}. Thanks to Theorem 2.12 pp. 21-22 from \cite{Ho98}, for each $a\in\R^d$ there exists a triplet $(\delta(a),\gamma(a),\nu(a))$ satisfying (H2($a$)) such that, $Lf(a):=T_{\chi_1,a}(\delta(a),\gamma(a),\nu(a))f$, for all $f\in\rmC_c^\infty(\R^d)$. It is clear that for any $a_0,a\in\R^d$ and $f\in\rmC_c^\infty(\R^d)$, the Lévy operator $L(a_0)$ defined by \eqref{eqLevOp} satisfies also
\begin{align*}
L(a_0)f(a)=T_{\chi_1,a}(\delta(a_0),\gamma(a_0),\nu_a(a_0)).
\end{align*}
Here and elsewhere $\nu_a(a_0)$ is the pushforward measure of $\nu(a_0)$ with respect to 
the translation $b\mapsto b-a_0+a$.
\begin{proof}[Proof of Theorem \ref{thmAppLT2}]
It suffices to prove that for any function $f_0\in\rmC^\infty_c(\R^d)$ and any sequence $a_n\in\R^d$ converging to $a_\infty\in\R^d$, the sequence $(T_nf_0(a_n)-f_0(a_n))/\eps_n$ converges to $Lf_0(a_\infty)$.
Thanks to Proposition \ref{propCVGLT} we have,  $\delta(a_n)\cv{n\to\infty}\delta(a_\infty)$, 
\[\forall f\in\rmC(\R^{d\Delta})\mbox{ vanishing  in a neighbourhood of $a_\infty$,}\int f(b)\nu(a_n,\d b)\cv{n\to\infty}\int f(b)\nu(a_\infty,\d b),\] 
and for all $1\leq i,j\leq d$
\[
\gamma_{ij}(a_n)+\int(\chi_i\chi_j)(a_n,b)\nu(a_n,\d b)\cv{n\to\infty}\gamma_{ij}(a_\infty)+\int(\chi_i\chi_j)(a_\infty,b)\nu(a_\infty,\d b).
\]
It is not difficult to deduce that, there exists $C\in\R_+$ such that, for all $n\in\N\cup\{\infty\}$ and $f\in\rmC^\infty_c(\R^d)$,
\[
\|L(a_n)f\|\leq C\|f\|\vee\max_{1\leq i\leq d}\|\partial_if\|\vee\max_{1\leq i,j\leq d}\|\partial^{2}_{ij}f\|.
\]
Hence $\sup_{n\in\N\cup\{\infty\}}\|L(a_n)f_0\|<\infty$.
Consider $b_\infty\in\R^d$,  a sequence $b_n\to b_\infty$ and a function $f\in\rmC(\R^{d\Delta})$ vanishing in a neighbourhood of $b_\infty$. By using
the first part of Lemma \ref{lmCvMes},
\begin{multline*}
\int f(b)\nu_{b_n}(a_n,\d b)=\int f(b-a_n+b_n)\nu(a_n,\d b)\\
\cv{n\to\infty}\int f(b-a_\infty+b_\infty)\nu(a_\infty,\d b)=\int f(b)\nu_{b_\infty}(a_\infty,\d b).
\end{multline*}
Hence by the second part of Theorem \ref{thmCVGLT}, $L(a_n)f$ converges uniformly on compact sets toward $L(a_\infty)f$, for all $f\in\rmC^\infty_c(\R^d)$. In particular, for each $\eps >0$ there exists an open neighbourhood $U$ of $a_\infty$ and $n_0\in\N$ 
such that
\[
\forall n\geq n_0,~\forall a\in U,\quad|L(a_n)f_0(a)-L(a_\infty)f_0(a_\infty)|\leq \eps.
\]
Let $\Pbf_n$ be the unique element of $\Mc(L(a_n))$ such that $\Pbf_n(X_0=a_n)=1$. Then, for all $n\geq n_0$
\begin{align*}
&\left|\frac{T_nf_0(a_n)-f_0(a_n)}{\eps_n}-L(a_\infty)f_0(a_\infty)\right|
=\left|\frac{\Ebf_n[f_0(X_{\eps_n})]-f_0(a_n)}{\eps_n}-L(a_\infty)f_0(a_\infty)\right|\\
&\hspace{1.5cm}=\left|\Ebf_n\frac{1}{\eps_n}\int_0^{\eps_n}\big(L(a_n)f_0(X_s)-L(a_\infty)f_0(a_\infty)\big)\d s\right|\\
&\hspace{1.5cm}\leq\Ebf_n\left[\1_{\{\tau^U<\eps_n\}}\frac{1}{\eps_n}\int_0^{\eps_n}
\big|L(a_n)f_0(X_s)-L(a_\infty)f_0(a_\infty)\big|\d s\right]\\
&\hspace{1.5cm}\quad + \Ebf_n\left[\1_{\{\tau^U\geq\eps_n\}}\frac{1}{\eps_n}\int_0^{\eps_n}\big|L(a_n)f_0(X_s)-L(a_\infty)f_0(a_\infty)\big|\d s\right]\\
&\hspace{1.5cm}\leq 2\Pbf_n(\tau^U<\eps_n)\sup_{m\in\N\cup\{\infty\}}\|L(a_m)f_0\| + \eps.
\end{align*}
We apply Lemma \ref{lmUQC} of uniform continuity along stopping times with  a compact neighbourhood $\Kc\subset U$ of $a_\infty$ and with $\Uc:=\R^d\times U\cup(\R^d\setminus \Kc)\times\R^d$. We deduce that
\[
\lim_{n\to\infty}\Pbf_n(\tau^U<\eps_n)=0.
\]
Hence
\[
\limsup_{n\to\infty}\left|\frac{T_nf_0(a_n)-f_0(a_n)}{\eps_n}-L(a_\infty)f_0(a_\infty)\right|\leq\eps,
\]
and we conclude by letting $\eps\to 0$.
\end{proof}

\section{Diffusion in a potential}

We recall that $\rmL^1_\loc(\R)$ denotes the space of locally Lebesgue integrable functions. 
A real continuous function $f$ is called locally absolutely continuous if its
distributional derivative $f^\prime$ belongs to $\rmL^1_\loc(\R)$.
We introduce the set of potential functions 
\[\Vc:=\left\{V:\R\to\R\text{ measurable}\mid\rme^{|V|}\in\rmL_\loc^1(\R)\right\}.\]
It is straightforward to prove that there exists a unique Polish topology on $\Vc$ such that a sequence $(V_n)_n$ in $\Vc$ converges to $V\in\Vc$ if and only if
\[
\forall M\in\R_+,\quad\lim_{n\to\infty}\int_{-M}^M|\rme^{V(a)}-\rme^{V_n(a)}|\vee|\rme^{-V(a)}-\rme^{-V_n(a)}|\,\d a=0.
\]
For a potential $V\in\Vc$, the operator 
\begin{equation}\label{difpotoper}
L^V:=\frac{1}{2}\rme^V\frac{\d}{\d a}\rme^{-V}\frac{\d}{\d a}
\end{equation}
is the set of couples $(f,g)\in\rmC_0(\R)\times\rmC(\R)$ such that $f$ and  $\rme^{-V}f^\prime$ are locally absolutely continuous and $g=\frac{1}{2}\rme^V(\rme^{-V}f^\prime)^\prime$. Let us notice that it is a particular case of the operator $\Dbf_m\Dbf_p^+$ described in \cite{Ma68}, pp. 21-22. Heuristically, the solutions of the martingale local problem associated to $L^V$ are solutions of the stochastic differential 
 equation \[\d X_t=\d B_t-\frac{1}{2}V^\prime(X_t)\d t,\]
 where $B$ is a standard Brownian motion.
\begin{proposition}[Diffusions on potential and random walks on $\Z$]\label{propDiffPot}~
\begin{enumerate}
\item\label{itPropDiffPot1} For any potential $V\in\Vc$, the operator $L^V$ is the generator of a locally Feller family.
\item\label{itPropDiffPot2} For any sequence of potentials $(V_n)_n$ in $\Vc$ converging to $V\in\Vc$ for the topology of $\Vc$, the sequence of operators $L^{V_n}$ converges to $L^V$, in the sense of 
the third statement 
of the convergence Theorem \ref{thmCvgLocFel}.
\item\label{itPropDiffPot3} For $(n,k)\in\N\times\Z$, let $q_{n,k}\in\R$ and $\eps_n>0$ be.
For all $n\in\N$, in accordance with Definition \ref{defDTSCMF}, let $(\Pbf^n_k)_k\in\Pcal(\Z^\N)^\Z$ be the unique discrete time locally Feller family such that
\[\Pbf^n_k(Y_1=k+1)=1-\Pbf^n_k(Y_1=k-1)=\frac{1}{\rme^{q_{n,k}}+1}.\]
Denote the sequence of potential in $\Vc$ by
\begin{align*}
V_n(a):=\sum_{k=1}^{\lfloor a/\eps_n\rfloor}q_{n,k}\1_{a\geq\eps_n} - \sum_{k=0}^{-\lfloor a/\eps_n\rfloor-1}q_{n,-k}\1_{a<0}\,.
\end{align*}
Let   $V$ be a potential in $\Vc$ and let $(\Pbf_a)_a$ be the locally Feller family associated with $L^V$.
Assume that $V_n$ converges to $V$ for the topology of $\Vc$, and $\eps_n\to 0$ as $n\to\infty$.
Then, for any sequence $\mu_n\in\Pcal(\Z)$ such that their pushforwards with respect to the mappings $k\mapsto\eps_nk$ converge to a probability measure $\mu\in\Pcal(\R)$, we have
\[
\law_{\Pbf^n_{\mu_n}}\left((\eps_nY_{\lfloor t/\eps_n^{2}\rfloor})_t\right)\cv[\Pcal\left(\Dloc(S)\right)]{n\to\infty}\Pbf_\mu.
\]
\end{enumerate}
\end{proposition}

Before proving this proposition, we give an important consequence 
concerning a random walk and a diffusion in random environment.  
For $n\in\N$, let $(\Omega^n,\Gc^n,\P^n)$ be a probability space and consider the random variables
\[
(q_{n,k})_k:\Omega^n\to\R^\Z,\quad (Z^n_k)_k:\Omega^n\to\Z^\N\quad\text{ and }\quad\eps_n:\Omega^n\to\R_+^*\,.
\]
Suppose that for any $n\in\N$ and $k\in\N$, $\P^n$-almost surely,
\begin{align*}
&\P^n\left(Z^n_{k+1}=Z^n_k+1\mid \eps_n,(q_{n,\ell})_{\ell\in\Z},(Z^n_\ell)_{0\leq\ell\leq k}\right)
=\frac{1}{\rme^{q_{n,Z_k}}+1}\\
&\P^n\left(Z^n_{k+1}=Z^n_k-1\mid \eps_n,(q_{n,\ell})_{\ell\in\Z},(Z^n_\ell)_{0\leq\ell\leq k}\right)
=\frac{1}{\rme^{-q_{n,Z_k}}+1}=1-\frac{1}{\rme^{q_{n,Z_k}}+1}.
\end{align*}
For any $n\in\N$ and $a\in\R$, denote the random potential in $\Vc$ by
\begin{equation}\label{Wn}
W_n(a):=\sum_{k=1}^{\lfloor a/\eps_n\rfloor}q_{n,k}\1_{a\geq\eps_n} - \sum_{k=0}^{-\lfloor a/\eps_n\rfloor-1}q_{n,-k}\1_{a<0}\,.
\end{equation}
Let $(\Omega,\Gc,\P)$ be a probability space and consider random variables
\[
W:\Omega\to\Vc\quad \text{ and }\quad Z:\Omega\to\Dloc(\R).
\]
Suppose that the conditional distribution of $Z$ with respect to $W$ satisfies, $\P$-almost surely,
\[
\law_\P\left(Z\mid W\right)\in\Mc(L^W).
\]
\begin{proposition}\label{propRWREtoDRE}
Suppose that $\eps_n$ converges in distribution to $0$, that $\eps_nZ^n_0$ converges in distribution to $Z_0$ and that $W_n$ converges in distribution to $W$ for the topology of $\Vc$. Then $(\eps_nZ^n_{\lfloor t/\eps_n^2\rfloor})_t$ converges in distribution to $Z$ for the local Skorokhod topology.
\end{proposition}
\begin{example}
1) Let $(q_k)_k$ be an i.i.d sequence of centred real random variables with finite variance $\sigma^2$ and suppose that $q_{n,k}=\sqrt{\eps_n}q_k$. Suppose also that $W$ is a  Brownian motion with variance $\sigma^2$. Then, by Donsker's theorem, $W_n$ 
given by \eqref{Wn} converges in distribution to $W$, so we can apply Proposition \ref{propRWREtoDRE} to deduce the convergence of the random walk in a random i.i.d. medium (introduced by Sinai 
in \cite{Si82}) to the diffusion in a Brownian potential (introduced in \cite{Br86}). Hence we recover Theorem 1 from \cite{Se00}, p. 295, without the hypothesis that the distribution of $q_0$ is compactly supported.

2) Fix deterministic $q\in\R$ and $\lambda\in\R_+^*$. Suppose that for each $n\in\N$, $(q_{n,k})_k$ is an i.i.d sequence of  random variables such that $\P^n(q_{n,k} = q)=1-\P^n(q_{n,k}=0)=\lambda\eps_n$. Suppose also that $W(a)=qN_{\lambda a}$, where $N$ is a standard Poisson process on $\R$. Then, it is classical (see for instance \cite{Ca97}), that $W_n$ given by \eqref{Wn} converges in distribution to $W$, so we can apply Proposition \ref{propRWREtoDRE} to deduce the convergence of Sinai's random walk to the diffusion in a Poisson potential. Hence we 
recover Theorem 2 from \cite{Se00}, p. 296.

3) More generally, suppose that for each $n\in\N$, $(q_{n,k})_k$ is an i.i.d sequence of
random variables. Likewise, suppose that $W_n$ given again by \eqref{Wn}, converges in distribution to a Lévy process $W$. We can apply Proposition \ref{propRWREtoDRE} to deduce the convergence of Sinai's random walk to the diffusion in a Lévy potential (introduced in \cite{Ca97}).
\end{example}
\begin{proof}[Proof of Proposition \ref{propRWREtoDRE}]
Let $F$ be a bounded continuous function from $\Dloc(\R)$ to $\R$. Define the bounded mapping
\[
G:\R\times\Vc\times\R_+\to\R
\]
as follows: for any $a\in\R$, $V\in\Vc$ and $\eps\in\R_+^*$, let $\Pbf^{a,V,\eps}\in\Pcal(\Z^\N)$ be the unique 
probability such that $\Pbf^{a,V,\eps}(Y_0=\lfloor a/\eps\rfloor)=1$ and, $\Pbf^{a,V,\eps}$-almost surely, for all $k\in\N$,
\begin{align*}
\Pbf^{a,V,\eps}\left(Y_{k+1}=Y_k+1\mid Y_0,\ldots, Y_k\right)&=1-\Pbf^{a,V,\eps}\left(Y_{k+1}=Y_k-1\mid Y_0,\ldots, Y_k\right)\\
&=\left.\int_{\eps Y_k-\eps}^{\eps Y_k}\rme^{V(a)}\d a\,\right/\,\int_{\eps Y_k-\eps}^{\eps Y_k+\eps}\rme^{V(a)}\d a.
\end{align*}
Let $\Pbf^{a,V,0}\in\Pcal(\Dloc(\R))$ be the unique element belonging to $\Mc(L^V)$ 
and starting from $a$. We set
\begin{align*}
G(a,V,\eps):=\Ebf^{a,V,\eps}\left[F\left((\eps Y_{\lfloor t/\eps^{2}\rfloor})_t\right)\right]
\quad\text{ and }\quad
G(a,V,0):=\Ebf^{a,V,0}\left[F(X)\right].
\end{align*}
By Proposition \ref{propDiffPot}, the mapping $G$ is continuous at every point of $\R\times\Vc\times\{0\}$. Thus,
\[
\E^n\left[G(\eps_nZ^n_0,W_n,\eps_n)\right]\cv{n\to\infty}\E\left[G(Z_0,W,0)\right].
\]
Hence
\begin{multline*}
\E^n\left[F\left((\eps_nZ^n_{\lfloor t/\eps_n^2\rfloor})_t\right)\right]
= \E^n\left[\E^n\left[F\left((\eps_nZ^n_{\lfloor t/\eps_n^2\rfloor})_t\right)\mid\eps_n,~Z^n_0,~(q_{n,\ell})_{\ell\in\Z}\right]\right]\\
=\E^n\left[G(\eps_nZ^n_0,W_n,\eps_n)\right]
\cv{n\to\infty}\E\left[G(Z_0,W,0)\right]
= \E\left[\E\left[F(Z)\mid Z_0,~W\right]\right]
= \E\left[F(Z)\right]\,.
\end{multline*}
Then $(\eps_nZ^n_{\lfloor t/\eps_n^2\rfloor})_t$ converges in distribution to $Z$.
\end{proof}

Before starting the proof of Proposition \ref{propDiffPot}, let us give a preliminary computation. Let $a_1,a_2\in\R$ be and let  $V:[a_1\wedge a_2,a_1\vee a_2]\to\R$ be a measurable function such that $\rme^{|V|}\in\rmL^1([a_1\wedge a_2,a_1\vee a_2])$. For any absolutely continuous function $f\in\rmC([a_1\wedge a_2,a_1\vee a_2],\R)$ such that $\rme^{-V}f^\prime$ is absolutely continuous and $g:=\frac{1}{2}\rme^V(\rme^{-V}f^\prime)^\prime$ is continuous, 
we have
\begin{align}
\nonumber f(a_2)
&=f(a_1)+\int_{a_1}^{a_2}f^\prime(b)\d b
=f(a_1)+\int_{a_1}^{a_2}\rme^{V(b)}\left((\rme^{-V}f^\prime)(a_1)+\int_{a_1}^b(\rme^{-V}f^\prime)^\prime(c)\d c\right)\d b\\
\label{eqPropDiffPot3}&=f(a_1)+\int_{a_1}^{a_2}\rme^{V(b)}\left((\rme^{-V}f^\prime)(a_1)+2\int_{a_1}^b\rme^{-V(c)}g(c)\d c\right)\d b\\
\label{eqPropDiffPot4}&\begin{aligned}
&=f(a_1)+(\rme^{-V}f^\prime)(a_1)\int_{a_1}^{a_2}\rme^{V(b)}\d b+2g(a_1)\int_{a_1}^{a_2}\int_{a_1}^b\rme^{V(b)-V(c)}\d c\,\d b\\
&\quad+2\int_{a_1}^{a_2}\int_{a_1}^b\rme^{V(b)-V(c)}(g(c)-g(a_1))\d c\,\d b.
\end{aligned}
\end{align}\begin{proof}[Proof of Proposition \ref{propDiffPot}]
\emph{Proof of \ref{itPropDiffPot1}.}
This part is essentially an application of the second chapter of \cite{Ma68}. 
For the sake of completeness we give here some details.
Let $h\in\rmC(\R,\R_+^*)$ be such that, for all $n\in\N$,
\begin{equation}\label{eqPropDiffPot2}\begin{aligned}
\inf_{n\leq|a|\leq n+1}h(a)\leq\frac{1}{n}&\Big[\,\int_n^{n+1}\int_0^a\rme^{V(b)-V(a)}\d b\,\d a\bigwedge\int_{n+1}^{n+2}\int_n^{n+1}\rme^{V(a)-V(b)}\d b\,\d a\\
&\bigwedge\int_{-n-1}^{-n}\int_a^0\rme^{V(b)-V(a)}\d b\,\d a\bigwedge\int_{-n-2}^{-n-1}\int_{-n-1}^{-n}\rme^{V(a)-V(b)}\d b\,\d a\Big]
\end{aligned}\end{equation}
The operator $hL^V$ coincides on $\rmC_0(\R)\times\rmC_0(\R)$ with the operator $D_mD_p^+\subset\rmC(\overline\R)\times\rmC(\overline\R)$ on the extended real line $\overline\R$, described  in \cite{Ma68}, pp. 21-22, where
\[
\d m(a):=\frac{2\rme^{-V(a)}}{h(a)}\d a\quad\text{ and }\quad\d p(a):=\rme^{V(a)}\d a.
\]
Applying \eqref{eqPropDiffPot2} we have 
\begin{align*}
&\int_0^\infty\int_0^a\d m(b)\d p(a)\geq\limsup_{n\to\infty}\int_{n+1}^{n+2}\int_{n}^{n+1}\d m(b)\d p(a)\geq\limsup_{n\to\infty}2n=\infty,\\
&\int_0^\infty\int_0^a\d p(b)\d m(a)\geq\limsup_{n\to\infty}\int_{n}^{n+1}\int_{0}^{a}\d p(b)\d m(a)\geq\limsup_{n\to\infty}2n=\infty,\\
&\int_{-\infty}^0\int_a^0\d m(b)\d p(a)\geq\limsup_{n\to\infty}\int_{-n-2}^{-n-1}\int_{-n-1}^{-n}\d m(b)\d p(a)\geq\limsup_{n\to\infty}2n=\infty,\\
&\int_{-\infty}^0\int_a^0\d p(b)\d m(a)\geq\limsup_{n\to\infty}\int_{-n-1}^{-n}\int_{a}^{0}\d p(b)\d m(a)\geq\limsup_{n\to\infty}2n=\infty.
\end{align*}
Thus according to the definition given in \cite{Ma68}, pp. 24-25, the boundary points $-\infty$ and $+\infty$ are natural. Thanks to Theorem 1 and Remark 2 p. 38 of \cite{Ma68}, $D_mD_p^+$ is the generator of a conservative Feller semi-group on $\rmC(\overline\R)$. Furthermore by steps 7 and 8 in \cite{Ma68},  pp. 31-32, 
\[D_mD_p^+f(-\infty)=D_mD_p^+f(+\infty)=0,\quad\forall f\in\rmD(D_mD_p^+),\] 
so that the operator 
\[D_mD_p^+\cap\rmC_0(\R)\times\rmC_0(\R)
=(hL^V)\cap\rmC_0(\R)\times\rmC_0(\R)\]
is the $\rmC_0\times\rmC_0$-generator of a Feller semi-group.
Hence, by \eqref{eqGenFF} and  \eqref{notationhL} we deduce that the operator
\[
\widetilde{L}:=\frac{1}{h}\overline{(hL^V)\cap\rmC_0(\R)\times\rmC_0(\R)}
\]
is the generator of a locally Feller family. Here  the closure is taken in $\rmC_0(\R)\times\rmC(\R)$, and it is clear that $\widetilde{L}\subset \overline{L^V}$. 
Thanks to  \eqref{eqPropDiffPot3} it is straightforward to obtain $L^V=\overline{L^V}$ and thanks to \eqref{eqPropDiffPot4} it is straightforward to obtain that $L^V$ satisfies the positive maximum principle. Thanks to Theorem \ref{thmExMP} we deduce the existence for the martingale local problem associated to $L^V$. Hence 
$L^V=\widetilde{L}$  is the generator of a locally Feller family.\\

\noindent
\emph{Proof of \ref{itPropDiffPot2}.}
Denote by $(\Pbf^n_a)_a$ and $(\Pbf^\infty_a)_a$ the locally Feller families associated,
respectively, to $L^{V_n}$ and  $L^V$.  By Theorem \ref{thmCvgLocFel} it is enough to prove that for each sequence $a_n\in\R$ converging to $a_\infty\in\R$, $\Pbf^n_{a_n}$ converges weakly to $\Pbf^\infty_{a_\infty}$ for the local Skorokhod topology.
Thanks to  Lemma \ref{lmMPtoSMP}, for  $M\in\N^*$, there exists $h_M\in\rmC(\R,[0,1])$ such that 
\[\{h_M\not = 0\}=(-2M,2M),\quad \{h_M=1\}=[-M,M]\] and, for all $n\in\N$, the martingale local problems associated to $h_ML^V$ and to $h_ML^{V_n}$ are well-posed.
For $n\in\N$ and $M\in\N^*$, denote by $(\Pbf^{n,M}_a)_a$ and $(\Pbf^{\infty,M}_a)_a$ the locally Feller families associated, respectively with $h_ML^{V_n}$ and $h_ML^V$.
For $n\in\N$, define the extension of $h_ML^{V_n}$:
\[
\widetilde{L_{n,M}}:=\Big\{(f,g)\in\rmC_0(\R)\times\rmC(\R)\,\big|\,
g=\frac{1}{2}h_M\rme^{V_n}(\rme^{-V_n}f^\prime)^\prime\mathds{1}_{(-2M,2M)}\Big\},
\]
where $f$ and $\rme^{-V_n}f^\prime$ are supposed locally absolutely continuous only on 
$(-2M,2M)$.
By \eqref{eqPropDiffPot4} it is straightforward to obtain that $\widetilde{L_{n,M}}$ satisfies the positive maximum principle, so by Theorem \ref{thmExMP}, $\widetilde{L_{n,M}}$ is a linear subspace of the generator of the family $(\Pbf^{n,M}_a)_a$.
We will prove that the sequence of operators $\widetilde{L_{n,M}}$ converges to the operator $h_ML^V$ in the sense of the third statement of  Theorem \ref{thmCvgLocFel}. Let $f\in\rmD(L)$ be and  define $f_n\in\rmC_0(\R)$ by
\begin{equation*}
f_n(a):=\left\{\begin{array}{ll}
f(a),\quad a\notin(-2M-n^{-1},2M+n^{-1})\\\\[-0.1cm]
\displaystyle f(0)+\int_0^a\rme^{V_n(b)}\Big[(\rme^{-V}f^\prime)(0)+2\int_0^b\rme^{-V_n(c)}L^Vf(c)\d c\Big]\d b,\quad a\in[-2M,2M],\end{array}\right.
\end{equation*}
with $f_n$ affine on $[-2M-n^{-1},-2M]$ and on $[2M,2M+n^{-1}]$.
Hence $f_n\in\rmD(\widetilde{L_{n,M}})$ and $\widetilde{L_{n,M}}f_n=h_ML^Vf$. We have
\[
\|f_n-f\|\leq\sup_{a\in[-2M,2M]}|f_n(a)-f(a)|+\sup_{\substack{2M\leq |a_1|,|a_2|\leq 2M+n^{-1}\\0\leq a_1a_2}}|f(a_2)-f(a_1)|.
\]
Since $f$ is continuous, the second supremum in the latter equation tends to $0$.
It is straightforward to deduce from \eqref{eqPropDiffPot3}, by using  the expression of $f_n$ and  the convergence $V_n\to V$, that 
\[\sup_{a\in[-2M,2M]}|f_n(a)-f(a)|\cv{n\to\infty}0.\]
Hence $\|f_n-f\|\to 0$ as $n\to\infty$, so the by Theorem \ref{thmCvgLocFel}:
\begin{align}\label{eqPropDiffPot6}
\Pbf^{n,M}_{a_n}\cv{n\to\infty}\Pbf^{\infty,M}_{a_\infty}.
\end{align}
By using Lemma \ref{lmEqGentoEqPro}, for all $M\in\N^*$ and $n\in\N\cup\{\infty\}$,
\begin{equation}\label{eqPropDiffPot7}
\law_{\Pbf^{n,M}_{a_n}}\left(X^{\tau^{(-M,M)}}\right) =\law_{\Pbf^n_{a_n}}\left(X^{\tau^{(-M,M)}}\right).
\end{equation}
At this level we use a result of localisation of the continuity contained in Lemma \ref{lemLocCont}.
Therefore, from \eqref{eqPropDiffPot6} and \eqref{eqPropDiffPot7}, letting $M\to\infty$, we deduce
\begin{align*}
\Pbf^n_{a_n}\cv{n\to\infty}\Pbf^\infty_{a_\infty}.
\end{align*}
\emph{Proof of \ref{itPropDiffPot3}.}
For $n\in\N$, define the continuous function $\varphi_n:\R\times\R\to\R_+$ given by 
\[
\varphi_n(a,h):=2\int_{a}^{a+h}\int_{a}^b\rme^{V_n(b)-V_n(c)}\d c\,\d b.
\]
For each $a\in\R$, it is clear that $\varphi_n(a,\cdot)$ is strictly increasing on $\R_+$ and  $\varphi_n(a,0)=0$. Furthermore, since 
$V_n$ is constant on the interval $\big[\eps_n\lceil a/\eps_n\rceil,\eps_n(\lceil a/\eps_n\rceil+1)\big)$,
\[
\varphi_n(a,2\eps_n)\geq2\int_{\eps_n\lceil a/\eps_n\rceil}^{\eps_n(\lceil a/\eps_n\rceil+1)}\int_{\eps_n\lceil a/\eps_n\rceil}^b\rme^{V_n(b)-V_n(c)}\d c\,\d b= \eps_n^2.
\]
Hence, there exists a unique $\psi_{1,n}(a)\in(0,2\eps_n]$ such that
\begin{align}\label{eqPropDiffPot8}
\varphi_n(a,\psi_{1,n}(a))=\eps_n^2.
\end{align}
Using the continuity of $\varphi_n$ and the compactness of $[0,2\eps_n]$, it is straightforward to obtain that $\psi_{1,n}$ is continuous. In the same manner, 
we may prove that, for each $a\in\R$, there exists a unique $\psi_{2,n}(a)\in(0,2\eps_n]$ such that
\begin{align}\label{eqPropDiffPot9}
\varphi_n(a,-\psi_{2,n}(a))=\eps_n^2,
\end{align}
and that $\psi_{2,n}$ is continuous. Introduce the continuous function $p_n:\R\to(0,1)$ given by
\begin{align}\label{eqPropDiffPot10}
p_n(a):=\left.\int_{a-\psi_{2,n}(a)}^{a}\rme^{V_n(b)}\d b\,\right/\int_{a-\psi_{2,n}(a)}^{a+\psi_{1,n}(a)}\rme^{V_n(b)}\d b,
\end{align}
and define a transition operator $T_n:\rmC_0(\R)\to\rmC_0(\R)$ by
\[
T_nf(a):= p_n(a)f(a+\psi_{1,n}(a))+(1-p_n(a))f(a-\psi_{2,n}(a)).
\]
According to Definition \ref{defDTSCMF}, let $(\widetilde{\Pbf}_a^n)_a\in\Pcal\left(\R^\N\right)^\R$ be the discrete time locally Feller family with transition operator $T_n$. 
For any $k\in\Z$, since $V_n$ is constant on $[\eps_nk,\eps_n(k+1))$ and on $[\eps_n(k-1),\eps_nk)$, we have
\begin{align*}
&\varphi_n(\eps_nk,\eps_n)=2\int_{\eps_nk}^{\eps_n(k+1)}\int_{\eps_nk}^b\d c\,\d b= \eps_n^2,\\
&\varphi_n(\eps_nk,-\eps_n)=2\int_{\eps_nk}^{\eps_n(k-1)}\int_{\eps_nk}^b\d c\,\d b= \eps_n^2,
\end{align*}
and therefore $\psi_{1,n}(\eps_nk)=\psi_{2,n}(\eps_nk)=\eps_n$. Furthermore
\[
p_n(\eps_nk):=\frac{\int_{\eps_n(k-1)}^{\eps_nk}\rme^{V_n(b)}\d b}{\int_{\eps_n(k-1)}^{\eps_n(k+1)}\rme^{V_n(b)}\d b}
=\frac{\eps_n\rme^{V_n(\eps_n(k-1))}}{\eps_n\rme^{V_n(\eps_n(k-1))}+\eps_n\rme^{V_n(\eps_nk)}}
=\frac{1}{1+\rme^{q_{n,k}}},
\]
hence for any $f\in\rmC_0(\R)$,
\[
T_nf(\eps_nk):= \frac{1}{1+\rme^{q_{n,k}}}f(\eps_n(k+1))+\frac{1}{1+\rme^{-q_{n,k}}}f(\eps_n(k-1)).
\]
We deduce that for any $\mu\in\Pcal(\Z)$ and $n\in\N$, $\law_{\Pbf^n_\mu}(\eps_nY)=\widetilde{\Pbf}_{\widetilde{\mu}}^n$\,, where $\widetilde{\mu}$ is the  pushforward measure of $\mu$ with respect to the mapping $k\mapsto\eps_nk$. 

We shall use Theorem \ref{thmCvgLocFel2} of convergence of discrete time Markov families.
If $f\in\rmD(L^V)$, we need to prove that there exists a sequence of continuous functions $f_n\in \rmC_0(\R)$ converging to $f$ such that $(T_nf_n-f_n)/\eps_n^2$ converges to $L^Vf$. By the second  part of Proposition \ref{propDiffPot}, there exists a sequence  of continuous functions $f_n\in\rmD(L^{V_n})$ such that $f_n$ converges to $f$ and $L^{V_n}f_n$ converges to $L^Vf$.
Applying \eqref{eqPropDiffPot4} to $f_n$ and $V_n$ and recalling \eqref{eqPropDiffPot8} and \eqref{eqPropDiffPot9}, we have for all $a\in\R$ and $n\in\N$,
\begin{align*}
f(a+\psi_{1,n}(a)) = &f(a)+ (\rme^{-V}f^\prime)(a)\int_{a}^{a+\psi_{1,n}(a)}\rme^{V(b)}\d b+\eps_n^2L^{V_n}f_n(a)\\
&+2\int_{a}^{a+\psi_{1,n}(a)}\int_{a}^b\rme^{V(b)-V(c)}(L^{V_n}f_n(c)-L^{V_n}f_n(a))\d c\,\d b,
\end{align*}
\begin{align*}
f(a-\psi_{2,n}(a)) = &f(a)- (\rme^{-V}f^\prime)(a)\int_{a-\psi_{2,n}(a)}^{a}\rme^{V(b)}\d b+\eps_n^2L^{V_n}f_n(a)\\
&+2\int_{a}^{a-\psi_{2,n}(a)}\int_{a}^b\rme^{V(b)-V(c)}(L^{V_n}f_n(c)-L^{V_n}f_n(a))\d c\,\d b.
\end{align*}
Hence by \eqref{eqPropDiffPot10}, for all $a\in\R$ and $n\in\N$,
\begin{align*}
&\Big|\frac{T_nf_n(a)-f_n(a)}{\eps_n^2}-L^{V_n}f_n(a)\Big|\\
&\hspace{2.5cm}\leq \frac{2p_n(a)}{\eps_n^2}\Big|\int_{a}^{a+\psi_{1,n}(a)}\int_{a}^b\rme^{V(b)-V(c)}(L^{V_n}f_n(c)-L^{V_n}f_n(a))\d c\,\d b\Big|\\
&\hspace{2.5cm}\quad + \frac{2(1-p_n(a))}{\eps_n^2}\Big|\int_{a}^{a-\psi_{2,n}(a)}\int_{a}^b\rme^{V(b)-V(c)}(L^{V_n}f_n(c)-L^{V_n}f_n(a))\d c\,\d b\Big|\\
&\hspace{2.5cm}\leq \sup_{|h|\leq2\eps_n}|L^{V_n}f_n(a+h)-L^{V_n}f_n(a)|.
\end{align*}
It is not difficult to deduce that $(T_nf_n-f_n)/\eps_n^2$ converges to $L^Vf$.
Finally, by the Theorem \ref{thmCvgLocFel2} of convergence of discrete time Markov families, for any sequence $\mu_n\in\Pcal(\Z)$ such that $\widetilde{\mu}_n$ converges to a probability measure $\mu\in\Pcal(\R)$, we have
\[
\law_{\Pbf^n_{\mu_n}}\left((\eps_nY_{\lfloor t/\eps_n^2\rfloor})_t\right)=\law_{\widetilde{\Pbf}_{\widetilde{\mu}_n}^n}\left((Y_{\lfloor t/\eps_n^2\rfloor})_t\right)\cv[\Pcal\left(\Dloc(S)\right)]{n\to\infty}\Pbf_\mu,
\]
where $\widetilde{\mu}_n$ are the pushforwards of $\mu_n$ with respect to the mappings $k\mapsto\eps_nk$.
\end{proof}

\appendix
\section*{Appendix}
\stepcounter{section}

We collect in this appendix several results already proved in \cite{GH117} having somehow
technical statements and used in the previous sections. We refer the interested reader to 
the paper \cite{GH117} for the introductory contexts and proofs  of each lemma. 

\begin{lemma}[cf. Lemma 3.6 in \cite{GH117}]\label{lmUQC}
Let $L_n,L_\infty\subset\rmC_0(S)\times\rmC(S)$ be such that $\rmD(L_\infty)$ is dense in $\rmC_0(S)$ and assume that, for any $f\in \rmD(L_\infty)$, there exist, for each $n$,  $f_n\in \rmD(L_n)$ such that $f_n\cv[\rmC_0]{n\to\infty} f$, $L_nf_n\cv[\rmC]{n\to\infty}L_\infty f$.
Consider $\Kc$ a compact subset of $S$ and $\Uc$ an open subset of $S\times S$ containing $\{(a,a)\,|\,a\in S\}$. For an arbitrary $(\Fc_{t+})_t$-stopping time $\tau_1$ we denote the $(\Fc_{t+})_t$-stopping time
\[
\tau(\tau_1):=\inf\left\{t\geq \tau_1\mid\{(X_{\tau_1},X_s)\}_{\tau_1\leq s\leq t}\not\Subset \Uc \right\}.
\]
Then for each $\eps>0$ there exist $n_0\in\N$ and $\delta>0$ such that: for any $n\geq n_0$, for any $\tau_1\leq\tau_2$, $(\Fc_{t+})_t$-stopping times, and for any $\Pbf\in\Mc(L_n)$ satisfying $\Ebf[(\tau_2-\tau_1)\mathds{1}_{\{X_{\tau_1}\in \Kc\}}]\leq\delta$, we have
\[
\Pbf(X_{\tau_1}\in \Kc,~\tau(\tau_1)\leq \tau_2)\leq\eps,
\]
with the convention $X_\infty:=\Delta$.
\end{lemma}

\begin{lemma}[cf. Proposition 4.15 in \cite{GH117}]\label{lmEqGentoEqPro}
Let $L_1,L_2\subset \rmC_0(S)\times\rmC(S)$ be such that $\rmD(L_1)=\rmD(L_2)$ is dense in $\rmC_0(S)$ and assume that the martingale local problems associated to $L_1$
and $L_2$ are well-posed. Let $\Pbf^1\in\Mc(L_1)$ and $\Pbf^2\in\Mc(L_2)$ be  two solutions of these problems having the same initial distribution and let  $U\subset S$ be an 
open subset. If 
\[
\forall f\in D(L_1),\;(L_2f)_{|U}=(L_1f)_{|U}\,,
\quad\mbox{ then }\quad
\law_{\Pbf^2}\left(X^{\tau^U}\right) = \law_{\Pbf^1}\left(X^{\tau^U}\right).
\]
\end{lemma}

\begin{lemma}[cf. Lemma 4.17 in \cite{GH117}]\label{lmMPtoSMP}
Let $U$ be an open subset of $S$ and $L$ be a subset of $\rmC_0(S)\times\rmC(S)$ with $\rmD(L)$ is dense in $\rmC(S)$. Assume that the martingale local problem associated to $L$ is well-posed. Then there exists a function $h_0\in\rmC(S,\R_+)$ satisfying $\{h_0\not =0\}=U$, such that for all $h\in\rmC(S,\R_+)$ with  $\{h\neq 0\}=U$ and  $\sup_{a\in U} (h/h_0)(a)<\infty$, the martingale local problem associated to $hL$ is well-posed. 
\end{lemma}

\begin{lemma}[cf. Lemma A.2 in \cite{GH117}]\label{lemLocCont}
Let $(U_m)_{m\in\N}$ be an increasing sequence of open subsets such that $S=\bigcup_mU_m$. 
For $n,m\in\N\cup\{\infty\}$, let $\Pbf^{n,m}\in\Pcal(\Dloc(S))$ be such that 
\begin{enumerate}
\item[i)] for each $m\in\N$, 
$\Pbf^{n,m}\cv{n\to\infty}\Pbf^{\infty,m}$, weakly for the  local Skorokhod topology,
\item[ii)]  for each $m\in\N$ and $n\in\N\cup\{\infty\}$,~~
$\law_{\Pbf^{n,m}}\left(X^{\tau^{U_m}}\right) =\law_{\Pbf^{n,\infty}}\left(X^{\tau^{U_m}}\right)$.
\end{enumerate}
Then 
$\Pbf^{n,\infty}\cv{n\to\infty}\Pbf^{\infty,\infty}$, weakly for the  local Skorokhod topology.
\end{lemma}

\bibliographystyle{alpha}
\bibliography{Loc_Fel.bib}

\begin{thebibliography}{BSW13}

\bibitem[Bro86]{Br86}
Th. Brox.
\newblock A one-dimensional diffusion process in a {W}iener medium.
\newblock {\em Ann. Probab.}, 14(4):1206--1218, 1986.

\bibitem[BSW13]{BSW13}
Bj\"orn B\"ottcher, Ren\'e Schilling, and Jian Wang.
\newblock {\em L\'evy matters. {III}}, volume 2099 of {\em Lecture Notes in
  Mathematics}.
\newblock Springer, Cham, 2013.
\newblock L\'evy-type processes: construction, approximation and sample path
  properties, With a short biography of Paul L\'evy by Jean Jacod, L\'evy
  Matters.

\bibitem[Car97]{Ca97}
Philippe Carmona.
\newblock The mean velocity of a {B}rownian motion in a random {L}\'evy
  potential.
\newblock {\em Ann. Probab.}, 25(4):1774--1788, 1997.

\bibitem[GH17a]{GH017}
Mihai Gradinaru and Tristan Haugomat.
\newblock Local {S}korokhod topology on the space of cadlag processes.
\newblock {\em arXiv: 1706.03669}, June 2017.

\bibitem[GH17b]{GH117}
Mihai Gradinaru and Tristan Haugomat.
\newblock Locally {F}eller processes and martingale local problems.
\newblock {\em arXiv: 1706.04880}, June 2017.

\bibitem[Hoh98]{Ho98}
Walter Hoh.
\newblock {\em Pseudo-Differential Operators Generating Markov Processes}.
\newblock Habilitationsschrift Universität Bielefeld, Bielefeld, 1998.

\bibitem[Kal02]{Ka02}
Olav Kallenberg.
\newblock {\em Foundations of modern probability}.
\newblock Probability and its Applications. Springer-Verlag, New York, second
  edition, 2002.

\bibitem[Kur11]{Ku11}
Thomas~G. Kurtz.
\newblock Equivalence of stochastic equations and martingale problems.
\newblock In {\em Stochastic analysis 2010}, pages 113--130. Springer-Verlag,
  Berlin, Heidelberg, 2011.

\bibitem[Man68]{Ma68}
Petr Mandl.
\newblock {\em Analytical treatment of one-dimensional {M}arkov processes}.
\newblock Die Grundlehren der mathematischen Wissenschaften, Band 151. Academia
  Publishing House of the Czechoslovak Academy of Sciences, Prague;
  Springer-Verlag New York Inc., New York, 1968.

\bibitem[Sei00]{Se00}
Paul Seignourel.
\newblock Discrete schemes for processes in random media.
\newblock {\em Probab. Theory Related Fields}, 118(3):293--322, 2000.

\bibitem[Sin82]{Si82}
Ya.~G. Sina\u\i.
\newblock The limit behavior of a one-dimensional random walk in a random
  environment.
\newblock {\em Teor. Veroyatnost. i Primenen.}, 27(2):247--258, 1982.

\bibitem[Str75]{St75}
Daniel~W. Stroock.
\newblock Diffusion processes associated with {L}\'evy generators.
\newblock {\em Z. Wahrscheinlichkeitstheorie und Verw. Gebiete},
  32(3):209--244, 1975.

\bibitem[SV06]{SV06}
Daniel~W. Stroock and S.~R.~Srinivasa Varadhan.
\newblock {\em Multidimensional diffusion processes}.
\newblock Classics in Mathematics. Springer-Verlag, Berlin, 2006.
\newblock Reprint of the 1997 edition.

\end{thebibliography}

\end{document}